\documentclass[10pt,a4paper]{article}
\usepackage[english]{babel}
\usepackage{amssymb}
\usepackage{amsfonts}
\usepackage{amsmath}
\usepackage{color}
\usepackage{amsthm}
\usepackage{graphicx}
\usepackage{upgreek}
\usepackage{multirow}
\usepackage[a4paper]{geometry}
\usepackage{algorithm}
\usepackage{algorithmic}
\usepackage{tabularx}
\usepackage{stmaryrd}
\usepackage{tikz}
\usepackage{appendix}
\usepackage{hyperref}
\usepackage{a4wide}

\newcommand{\be}{\begin{equation}}
\newcommand{\ee}{\end{equation}}
\newcommand{\bi}{\begin{itemize}}
\newcommand{\ei}{\end{itemize}}
\newcommand{\norm}[1]{\left\Vert{#1}\right\Vert} 

\newtheorem{proposition}{Proposition}
\newtheorem{definition}{Definition}
\newtheorem{corollaire}{Corollary}
\newtheorem{lemme}{Lemma}

\newtheorem{theoreme}{Theorem}
\newtheorem{remarque}{Remark}

\newcommand{\prox}{\text{prox}}
\newcommand{\N}{\mathbb{N}}

\newcommand{\R}{\mathbb{R}}

\newcommand{\xn}{(x_n)_{n\in\N}}

\def\arg{\textup{arg}\,}
\newcolumntype{C}{>{\centering}X}
\newcommand{\off}[1]{}

\date{}
\begin{document}
\title{Strong Convergence of FISTA Iterates under H\"olderian and Quadratic Growth Conditions}
\author {J.-F. Aujol\footnote{Univ. Bordeaux, Bordeaux INP, CNRS, IMB, UMR 5251, F-33400 Talence, France}
\and 
C. Dossal\footnote{IMT, Univ. Toulouse, INSA Toulouse, Toulouse, France}
\and
H. Labarri\`ere\footnote{MaLGa, DIBRIS, Universit\`a di Genova, Genoa, Italy}
\and 
A. Rondepierre\footnotemark[2] \footnote{LAAS, Univ. Toulouse, CNRS, Toulouse, France}\\
}

\date{\today}
\maketitle
\begin{abstract}
    Introduced by Beck and Teboulle in \cite{beck2009fast}, FISTA (for Fast Iterative Shrinkage-Thresholding Algorithm) is a first-order method widely used in convex optimization. Adapted from Nesterov's accelerated gradient method for convex functions \cite{nesterov1983method}, the generated sequence guarantees a decay of the function values of $\mathcal{O}\left(n^{-2}\right)$ in the convex setting. We show that for coercive functions satisfying some local growth condition (namely a H\"olderian or quadratic growth condition), this sequence strongly converges to a minimizer. This property, which has never been proved without assuming the uniqueness of the minimizer, is associated with improved convergence rates for the function values. The proposed analysis is based on a preliminary study of the Asymptotic Vanishing Damping system introduced by Su et al. in \cite{su2014differential} to model Nesterov's accelerated gradient method in a continuous setting. Novel improved convergence results are also shown for the solutions of this dynamical system, including the finite length of the trajectory under the aforementioned geometry conditions.
\end{abstract}

\section{Introduction}

Fast Iterative Shrinkage-Thresholding Algorithm (FISTA) is a well-known scheme introduced by Beck and Teboulle in \cite{beck2009fast} for the minimization of convex composite functions. Considering a convex function $F:\mathcal{H}\rightarrow \R$ where $\mathcal{H}$ is a Hilbert space, $F$ is called composite if it can be written $F=f+h$ where $f$ is a convex differentiable function having a $L$-Lipchitz gradient and $h$ is is a proper lower semicontinuous (l.s.c.) convex function. 

This method uses inertia to achieve acceleration, based on the ideas proposed by Nesterov in the convex setting \cite{nesterov1983method}. While the classical proximal gradient method (also called Forward-Backward \cite{combettes2005signal}) guarantees a decrease of the error of order $\mathcal{O}\left(n^{-1}\right)$, FISTA builds a sequence $\xn$ which ensures that if $F$ is a convex composite function, then
\begin{equation}\label{eq:rate_FISTA}
    F(x_n)-F^*\leqslant \frac{2L\|x_0-x^*\|^2}{(n+1)^2},
\end{equation}
for any minimizer $x^*$ of $X^*$ where $F^*=\min_{x\in\mathcal{H}}F(x)$. 
The question of the convergence of FISTA iterates remained unanswered for a few years before Chambolle and D. show in \cite{chambolle2015convergence} that for a slightly modified inertial term depending on a non negative real number $\alpha>3$, the sequence $\xn$ weakly converges to a minimizer of $F$. This $\mathcal{O}\left(n^{-2}\right)$ rate can be improved to a $o\left(n^{-2}\right)$ rate for this variation of FISTA proposed by Chambolle and D. as demonstrated by Attouch and Peypouquet in \cite{attouch2016rate} but no first order method can guarantee a decrease of the error faster than this rate for this class of functions as shown in \cite{nesterov2003introductory}. 

Better convergence guarantees can be proven by making stronger assumptions on the function $F$ and by conveniently adjusting the inertial parameter. Su et al. \cite{su2014differential} show that FISTA iterates can achieve a rate of $F(x_n)-F^*=\mathcal{O}\left(n^{-3}\right)$ for strongly convex functions. Attouch and Cabot \cite{attouch2018fast} improve this result as they prove that the error decreases as $\mathcal{O}\left(n^{-\frac{2\alpha}{3}}\right)$ for any $\alpha>0$ as long as $F$ has a strong minimum, i.e. $F$ has a unique minimizer $x^*$ and a \textbf{global quadratic growth}:
\begin{equation}
    \exists \mu>0,~\forall x\in\mathcal{H}, ~\frac{\mu}{2}\|x-x^*\|^2\leqslant F(x)-F^*. \label{eq:G2mu_intro}
\end{equation}
An additional flatness condition which requires the differentiability of $F$ allows to strengthen this convergence guarantee as shown in \cite{apidopoulos2021convergence,aujol2023fista}. In \cite{apidopoulos2021convergence}, Apidopoulos et al. give an improved convergence rate of the error under the aforementioned flatness condition, a uniqueness assumption on the minimizer $x^*$ of $F$ and a \textbf{H\"olderian error bound} hypothesis:
\begin{equation}
    \exists \gamma>2,~\exists K>0,~\forall x\in B(x^*,\varepsilon), ~K\|x-x^*\|^\gamma\leqslant F(x)-F^*.\label{eq:holder_intro}
\end{equation}

The works mentioned above mainly focus on finding the fastest convergence rate and since every improved result relies on the hypothesis that $F$ has a unique minimizer $x^*$, the strong convergence of FISTA iterates is actually trivial (under these hypotheses, $F(x_n)-F^*\rightarrow0$ implies that $\|x_n-x^*\|\rightarrow0$).

This observation is also true when considering the study of the corresponding ordinary differential equation (ODE) i.e. Asymptotic Vanishing Damping system (AVD) defined by
\begin{equation}\label{eq:avd_intro}
    \ddot{x}(t)+\frac{\alpha}{t}\dot{x}(t)+\nabla F(x(t))=0.\tag{AVD}
\end{equation}
Introduced by Su et al. in \cite{su2014differential} as a system which can be discretized to recover Nesterov's accelerated gradient method, this ODE shares most of its convergence properties with FISTA iterates. Several papers (see \cite{attouch2018convergence,aujol2019optimal,aujol2023fista,Luo2023}) are devoted to its analysis under geometry assumptions and most of the fast convergence results require $F$ to have a unique minimizer $x^*$ which automatically guarantees that $\|x(t)-x^*\|\rightarrow0$.

In this paper, we analyse theoretically the behavior of FISTA iterates and its corresponding ODE under H\"olderian and quadratic growth assumptions without any hypothesis on the uniqueness of the minimizer. Indeed in this geometrical setting, the strong convergence of the iterates $(x_n)_{n\in\N}$ (resp the trajectory $x(\cdot)$) is no longer a consequence of the decay of $\left(F(x_n)\right)_{n\in\N}$ (resp $F(x(\cdot))$) but a consequence of the bounds of $\left(\norm{x_n-x_{n-1}}\right)_{n\in\N}$ (resp $\norm{\dot x(\cdot)}$). The main contributions are the following :
\begin{enumerate}
    \item Strong convergence of FISTA iterates for functions having a local H\"olderian growth \eqref{eq:holder_intro} with parameter $\gamma>2$ (for a well-chosen inertial parameter). In addition, we prove that the error $F(x_n)-F^*$ decreases as $\mathcal{O}\left(n^{-\frac{2\gamma}{\gamma-2}}\right)$.
    \item Strong convergence of FISTA iterates for functions having a quadratic growth \eqref{eq:G2mu_intro} (for a well-chosen inertial parameter) and non-asymptotic bound on the error if this assumption is global. We recover the convergence rate proved if $F$ has a unique minimizer i.e.
    $$F(x_n)-F^*=\mathcal{O}\left(n^{-\frac{2\alpha}{3}}\right),$$
    for $\alpha$ sufficiently large.
    \item Finite trajectory of the solution of \eqref{eq:avd_intro} under H\"olderian or quadratic growth without a uniqueness assumption on the minimizers of $F$. We show that if the set of minimizers $X^*$ is sufficiently regular, the error along the trajectories decreases respectively as $\mathcal{O}\left(t^{-\frac{2\gamma}{\gamma-2}}\right)$ or $\mathcal{O}\left(t^{-\frac{2\alpha}{3}}\right)$ for the aforementioned assumptions if $\alpha$ is sufficiently large.
\end{enumerate}

The paper is organized as follows. Section \ref{sec:pre_soa} presents key geometry concepts used in the paper before giving an overview of the literature on FISTA and the Asymptotic Vanishing Damping system. The main results on the strong convergence of FISTA iterates are then stated and discussed in Section \ref{sec:Holderian}. Section \ref{sec:AVD} contains the analogous convergence results obtained for the trajectories of the Asymptotic Vanishing Damping system. The proofs of the main theorems are given in Section \ref{sec:proofs_main} while the other demonstrations are postponed to Appendix \ref{app:1} and Appendix \ref{app:2}.

\section{Preliminaries and State of the Art}\label{sec:pre_soa}
Let $\mathcal{H}$ be a Hilbert space. This work focuses on the class $\mathcal{C}$ of composite functions defined by:
\begin{definition}\label{def:Composite}
Let $\mathcal{C}$ be the class of convex functions $F$ defined from $\mathcal{H}$ to $\R\cup \{+\infty\}$ such that
$F=f+h$, where $f$ is a convex differentiable function having a $L$-Lipschitz gradient, and $h$ is a convex function whose proximal operator is known. The set of minimizers $X^*$ of $F$ is non-empty but not necessarily reduced to one point.    
\end{definition}
This set $\mathcal{C}$ depends on the non negative real number $L$, but to lighten the notation and because there is no ambiguity, we choose the simple notation $\mathcal{C}$. 
\subsection{Geometry of convex functions}\label{sec:geo}


In this paper we consider the general class of convex composite functions satisfying some growth condition in the neighborhood of their sets of minimizers:
\begin{definition}[\textit{Local growth conditions}]
Let $F:\mathcal{H}\rightarrow \R\cup \{+\infty\}$ be a proper lower semicontinuous convex function with a non-empty set of minimizers $X^*$. Let $F^*=\min_{x\in\mathcal{H}} F(x)$. The function $F$ is said to satisfy a Hölderian growth condition $\mathcal{G}^\gamma_\text{loc}$ for some $\gamma>2$ if  there exist $K>0$ and $\varepsilon>0$ such that for all $x\in \mathcal H$ satisfying $d(x,X^*) \leqslant\varepsilon$, we have:
\begin{equation}
K d(x,X^*)^\gamma\leqslant F(x)-F^*.
\label{eq:Lojasiewicz_dis_flat}
\end{equation}
Moreover, the function $F$ satisfies a local quadratic growth condition $\mathcal{G}_{\mu,\text{loc}}^2$ for some $\mu>0$ if there exists $\varepsilon>0$ such that for all $x\in \mathcal H$ satisfying: $d(x,X^*) \leqslant\varepsilon$, we have:
\begin{equation}
\frac{\mu}{2}d(x,X^*)^2\leqslant F(x)-F^*.
\label{eq:Lojasiewicz_dis_local}
\end{equation}
\end{definition}
In the context of finite-time analysis, we also introduce the global version of these growth conditions:
\begin{definition}[\textit{Global growth conditions}]
Let $F:\mathcal{H}\rightarrow \R\cup \{+\infty\}$ be a proper lower semicontinuous convex function with a non-empty set of minimizers $X^*$. Let $F^*=\min_{x\in\mathcal{H}} F(x)$. The function $F$ satisfies the growth condition $\mathcal{G}^\gamma$ for some $\gamma>2$ if there exists $K>0$ such that:
\begin{equation}
\forall x\in\mathcal{H},\quad K d(x,X^*)^\gamma\leqslant F(x)-F^*.
\label{eq:Lojasiewicz_dis_flat_global}
\end{equation}
Moreover, the function $F$ satisfies a quadratic growth condition $\mathcal{G}_\mu^2$ for some $\mu>0$ if:
\begin{equation}
\forall x\in\mathcal{H},\quad \frac{\mu}{2}d(x,X^*)^2\leqslant F(x)-F^*.
\label{eq:Lojasiewicz_dis}
\end{equation}
\end{definition}
The growth conditions $\mathcal{G}^\gamma_\text{loc}$ ($\gamma\geqslant 2$) can be seen as sharpness assumptions on the function $F$ characterizing functions behaving at least as $\|\cdot\|^\gamma$ in the neighborhood of their minimizers. In the convex setting, the class of functions satisfying some growth condition is a subclass of the functions having a \L ojasiewicz property \cite{loja63,loja93}, a key tool for the mathematical analysis of continuous and discrete dynamical systems. Initially introduced to prove the convergence of the trajectories for the gradient flow of analytic functions, an extension to nonsmooth functions has been proposed by Bolte et al. in \cite{Bolte2006,Bolte2007loja}:
\begin{definition}[The \L ojasiewicz property]
\label{def_loja}
Let $F:\mathcal{H}\rightarrow \R\cup \{+\infty\}$ be a proper lower semicontinuous convex function with a non-empty set of minimizers $X^*$. Let $F^*=\min_{x\in\mathcal{H}} F(x)$.
The function $F$ has a \L ojasiewicz property if for any minimizer $x^*$, there exist $\theta\in [0,1)$, $c>0$, $\varepsilon >0$ such that:
\begin{equation}
\forall x\in B(x^*,\varepsilon),~c\left(F(x)-F^*\right)^\theta \leqslant d(0,\partial F(x)).\label{Loja}
\end{equation}
\end{definition}

Let us finally introduce the notion of flatness characterizing differentiable functions that are at least as flat as $\|\cdot\|^\gamma$ with $\gamma>1$:
\begin{equation}\label{eq:flatness}\tag{$\mathcal{F}_\gamma$}
    \forall x^*\in X^*,\quad\forall x\in\mathcal{H},\quad F(x)-F^*\leqslant\frac{1}{\gamma}\left\langle\nabla F(x),x-x^*\right\rangle,
\end{equation}
where $F^*=\min_{x\in\mathcal{H}} F(x)$. Note that if $F$ is convex, then it satisfies \eqref{eq:flatness} for $\gamma=1$. This notion is recalled here to enable latter comparisons, particularly with the convergence results presented in \cite{apidopoulos2021convergence}.

To conclude this section, observe that in the context of local growth assumptions, the convergence of the sequence of $(F(x_n)-F^*)_{n\in \N}$ to $0$ does not trivially imply the convergence of a given sequence of iterates $(x_n)_{n\in \N}$ to the set of minimizers $X^*$. The coercivity of $F$ is needed to conclude:
\begin{lemme}\label{lem_cv_locale}
Let $F\in\mathcal{C}$ be a coercive function satisfying a local growth condition $\mathcal G_{loc}^\gamma$ for some real parameters $\gamma\geqslant 2$ and $K>0$. Let $(x_n)_{n\in\N}$ be a sequence of iterates generated by a given algorithm $\mathcal A$.

If the sequence $(F(x_n)-F^*)_{n\in \N}$ converge to $0$, then $\left(d(x_n,X^*)\right)_{n\in \mathbb N}$ converges to $0$ and:
$$\exists N\in \mathbb N,~\forall n\geqslant N,~F(x_n)-F^*\geqslant Kd(x_n,X^*)^\gamma.$$
\end{lemme}

\begin{proof}
Assume that the sequence $\left(d(x_n,X^*)\right)_{n\in\N}$ does not converge to $0$. Thus, there exists $\varepsilon>0$ and a non-decreasing function $\phi:\N\rightarrow\N$ such that the sub-sequence $(x_{\phi(n)})_{n\in\N}$ satisfies:
\begin{equation*}
\forall n\in \N,~    d(x_{\phi(n)},X^*)\geqslant \varepsilon.
\end{equation*}
Since the sequence $\left(F(x_n)-F^*\right)_{n\in\N}$ is assumed to converge to $0$ , it is also bounded. Combined with the coercivity of $F$, this implies that the sequence $(x_n)_{n\in\N}$ is bounded too. Therefore, there exists a closed bounded set $C$ containing $X^*$ such that 
\begin{equation}
    \left\{x_n,~n\in\N\right\}\subset C.
\end{equation}
Let $K_\varepsilon=C\cap\left\{x\in\mathcal{H},~d(x,X^*)\geqslant\varepsilon\right\}$. By construction, $K_\varepsilon$ is a weakly compact subset of $\mathcal{H}$ and $K_\varepsilon\cap X^*=\emptyset$. Moreover, for all $n\in\N$, we have $x_{\phi(n)}\in K_\varepsilon$ so there exists a weakly  convergent sub-sequence $\left(x_{\psi\circ\phi(n)}\right)_{n\in\N}$ whose weak limit denoted by $\tilde x$ belongs to $K_\varepsilon$ and thus $\tilde x\notin X^*$.

Consequently, since $F$ is convex and lower semi-continuous (we remind the reader that when $F$ is convex, then $F$ is weak lsc if and only if $F$ is strong lsc, see e.g. \cite{Brezis}),we have
\begin{equation}
  \liminf{  F\left(x_{\psi\circ\phi(n)}\right)} - F^*\geqslant F(\tilde x) -F^*.
\end{equation}
Since the whole sequence $\left(F(x_n)-F^*\right)_{n\in\N}$ tends to $0$ when $n\rightarrow +\infty$, and since $F(\tilde x)  \geqslant F^*$, it implies that $F(\tilde x)-F^*=0$ which is impossible since $\tilde x\notin X^*$. Thus, the sequence $\left(d(x_n,X^*)\right)_{n\in\N}$ converges to $0$ as $n\rightarrow+\infty$.
\end{proof}

This technical lemma will be useful throughout the paper to establish new convergence rates for the class of composite functions satisfying certain growth conditions, without assuming the uniqueness of the minimizer.

\subsection{FISTA and its variants}

To solve the minimization problem
\begin{equation}
    \min_{x\in\mathcal{H}}F(x),
\end{equation}
where $F$ is a convex composite function in the class $\mathcal C$ (see Definition \ref{def:Composite})
, a classical algorithm is the \textbf{Proximal Gradient method} also called \textbf{Forward-Backward} \cite{combettes2005signal}. Before defining properly this scheme, it is necessary to introduce the notion of proximal operator. Considering $h:\mathcal{H}\rightarrow\R\cup\{+\infty\}$ a proper lower semicontinuous convex function, its proximal operator denoted $\prox_h$ is defined for all $x\in\mathcal{H}$ as
\begin{equation}\label{eq:prox}
    \prox_h(x)=\arg\min\limits_{y\in\mathcal{H}}h(y)+\frac{1}{2}\|x-y\|^2.
\end{equation}
Given an initialization $x_0\in\mathcal{H}$, the iterates of the Proximal Gradient method are defined as
\begin{equation}\label{eq:FB}
    \forall n\in\N,\quad x_{n+1} = \prox_{sh}\left(x_n-s\nabla f(x_n)\right),
\end{equation}
where the step size $s>0$ should be chosen smaller than $\frac{1}{L}$ to ensure that $F(x_n)-F^*=\mathcal{O}\left(n^{-1}\right)$.\\
In 2009 Beck and Teboulle introduce in \cite{beck2009fast} the Fast Iterative Shrinkage-Thresholding Algorithm (FISTA) for the same class of functions. While the Proximal Gradient method is the composite extension of the Gradient Descent method in the differentiable setting, FISTA is a generalization of Nesterov's accelerated gradient method for convex functions \cite{nesterov1983method}. Indeed, the iterates of FISTA are defined in the following way:
\begin{equation}\label{eq:prox_grad_accel}
    x_0\in\mathcal{H},\quad\forall n\in\N,\quad\left\{\begin{aligned}
        &y_n=x_n+\alpha_n\left(x_n-x_{n-1}\right)\\
        &x_{n+1}=\prox_{sh}\left(y_n-s\nabla f(y_n)\right),
    \end{aligned}\right.
\end{equation}
where $x_{-1}=x_0$ and the sequence $\left(\alpha_n\right)_{n\in\N}$ is that defined by Nesterov in \cite{nesterov1983method} as:
\begin{equation}
\label{eq:choix_Nesterov_alpha}
\alpha_0=0,\quad\forall n\in\N,\left\{\begin{aligned}
    t_{n+1}&=\frac{1+\sqrt{1+4t_n^2}}{2}\\
    \alpha_{n+1} &= \frac{t_n-1}{t_{n+1}},
\end{aligned}\right.
\end{equation}
where $t_0=1$. The authors prove that $F(x_n)-F^*=\mathcal{O}\left(n^{-2}\right)$ for $s\in\left(0,\frac{1}{L}\right)$ (and in particular \eqref{eq:rate_FISTA} for $s=\frac{1}{L}$). Although this convergence rate reveals a significant improvement over Proximal Gradient method, the authors do not show the weak convergence of the iterates.

This property of the sequence $\left(x_n\right)_{n\in\N}$ is proved by Chambolle and Dossal in \cite{chambolle2015convergence} for a slightly different version of FISTA, choosing $\left(\alpha_n\right)_{n\in\N}$ as $\alpha_n=\frac{n}{n+\alpha}$ with $\alpha>3$. Attouch and Peypouquet show in \cite{attouch2016rate} that this choice for $\alpha$ ensures that $F(x_n)-F^*=o\left(n^{-2}\right)$. 

\begin{remarque}
    The sequence $\left(\alpha_n\right)_{n\in\N}$ introduced by Chambolle and Dossal (defined as $\alpha_n=\frac{n}{n+\alpha}$ with $\alpha>3$) and that given by Nesterov (i.e. \eqref{eq:choix_Nesterov_alpha}) have a similar behavior when $\alpha=3$. In practice, the Chambolle-Dossal formulation is more convenient to draw a parallel with the continuous setting (see Section \ref{sec:AVD}) and to obtain improved convergence properties under additional geometry assumptions, while the Nesterov formulation facilitates the implementation of linesearch strategies.
\end{remarque}

Note that Kim and Fessler introduce Optimized Gradient Method in \cite{Kim2016} (and a proximal version in \cite{kim2018adaptive}) which also ensures a decrease of the error of order $\mathcal{O}\left(n^{-2}\right)$ but with a tightest and optimal bound in the differentiable case. 

\begin{remarque}[Why so many names ?]
    When introduced by Beck and Teboulle, FISTA is presented as an accelerated version of Iterative Shrinkage-Thresholding algorithms \cite{daubechies2004iterative} (ISTA) which are methods solving problems of the form:
    \begin{equation*}
        \min_{x\in\mathcal{H}}F(x):=\|Ax-b\|^2+\lambda\|x\|_1.
    \end{equation*}
    The appellation ISTA (and consequently FISTA) comes from the fact that the proximal operator of $\|\cdot\|_1$ is the soft-thresholding operator. However, the function $F$ defined in that way only belongs to a subclass of composite convex functions that FISTA can actually minimize.

    This confusion may explain the numerous names given to FISTA such as Nesterov's Accelerated Forward-Backward \cite{attouch2016rate}, Accelerated Proximal Gradient Descent \cite{li2015accelerated} or Inertial Forward-Backward \cite{attouch2016rate}. It also occurs that FISTA refers to \eqref{eq:prox_grad_accel} where the sequence $\left(\alpha_n\right)_{n\in\N}$ is set constant in time equal to $\alpha\in(0,1)$ (see \cite{li2022linear}), a method also called V-FISTA by Beck in \cite{beck2017first}.
\end{remarque}

\subsection{Convergence under additional geometry assumptions}\label{Sec:ConvGeo}
In this section, we give an overview of the known convergence properties of the Proximal Gradient Methods and of the Chambolle-Dossal formulation of FISTA for convex composite functions satisfying an additional growth assumption. 
\paragraph{Proximal Gradient Method under geometry assumptions}

The convergence of Proximal Gradient Method has been studied under several growth conditions in particular by Garrigos et al. in \cite{garrigos2022convergence}. In this paper, the authors prove that the iterates of the Proximal Gradient Method converge strongly to a minimizer of $F$ if the function is $p$-\L{}ojasiewicz with $p\geqslant1$ without any uniqueness assumption on the set of minimizers. The \L{}ojasiewicz property can be linked to the growth assumptions stated in Section \ref{sec:geo} and the strong convergence result holds if $F$ satisfies $\mathcal{G}^2_\mu$ or $\mathcal{G}^\gamma$. Moreover, if $F$ satisfies $\mathcal{G}^2_\mu$, then 
\begin{equation*}
F(x_n)-F^*=\mathcal{O}\left(e^{-\frac{\mu}{4L}n}\right)
\end{equation*}
and if $F$ has an H\"olderian growth i.e. $\mathcal{G}^\gamma$ with $\gamma>2$ then
\begin{equation*}
F(x_n)-F^*=\mathcal{O}\left(n^{-\frac{\gamma}{\gamma-2}}\right).
\end{equation*}

\paragraph{FISTA under geometry assumptions}

As stated previously, it is well known (see \cite{chambolle2015convergence,su2014differential}) that in a convex setting the iterates of the Chambolle-Dossal formulation of FISTA i.e. $\alpha_n=\frac{n}{n+\alpha}$, satisfy:
\begin{equation}\label{eq:cv_FISTA_cv}
F(x_n)-F^*\leqslant \frac{(\alpha-1)^2\|x_0-x^*\|^2}{2s(n+\alpha-2)^2},
\end{equation}
for any $x^*\in X^*$ as long as $s\leqslant \frac{1}{L}$ and $\alpha\geqslant3$. The following works show that additional assumptions on $F$ allow to guarantee better convergence properties. The paragraph is summarized in Table \ref{tab:fista_nu}.

First, Su, Boyd and Candès show in \cite{su2014differential} that this rate can be improved to $\mathcal{O}\left(n^{-3}\right)$ for strongly convex functions if $\alpha\geqslant \frac{9}{2}$. Attouch and Cabot strengthen this result in \cite{attouch2018convergence} by proving that $F(x_n)-F^*=\mathcal{O}\left(n^{-\frac{2\alpha}{3}}\right)$ for $\alpha>0$ when $F$ has a strong minimizer, i.e. $F$ has a quadratic growth $\mathcal{G}^2_\mu$ and a unique minimizer. 
The understanding of FISTA in this setting is then enhanced by Aujol et al. in \cite{aujol2023fista} as the authors provide non-asymptotical results enlightening the dependency in $\alpha$. Apidopoulos et al. also give improved guarantees for functions having a H\"olderian and quadratic growth in \cite{apidopoulos2021convergence}. 

{\renewcommand{\arraystretch}{1.5}
\begin{table}[H]
\centering
 \begin{tabularx}{\linewidth}{|>{\centering}X|>{\centering}X|>{\centering}X|>{\centering}X|}
  \hline
  \textbf{Reference} & \textbf{Assumption on $F$} & \textbf{Parameter range} & \textbf{Convergence rate of $F(x_n)-F^*$} \tabularnewline
  \hline
  Su et al. \cite{su2014differential} & Strong convexity & $\alpha\geqslant \frac{9}{2}$ & $\mathcal{O}\left(n^{-3}\right)$\tabularnewline\hline
  Attouch, Cabot \cite{attouch2018convergence} & $\mathcal{G}^2_\mu$ and uniqueness of the minimizer &$\alpha>0$ & $\mathcal{O}\left(n^{-\frac{2\alpha}{3}}\right)$\tabularnewline\hline
  Apidopoulos et al. \cite{apidopoulos2021convergence}\\Aujol et al. \cite{aujol2023fista} & $\mathcal{F}_\gamma$ and $\mathcal{G}^2_\mu$, $\gamma\geqslant1$\\Uniqueness of the minimizer & $\alpha\geqslant 1+\frac{2}{\gamma}$ & $\mathcal{O}\left(n^{-\frac{2\alpha\gamma}{\gamma+2}}\right)$\tabularnewline
  \hline
  Apidopoulos et al. \cite{apidopoulos2021convergence}& $\mathcal{F}_{\gamma_1}$ and $\mathcal{G}^{\gamma_2}$ where $\gamma_2\geqslant\gamma_1>2$\\Uniqueness of the minimizer & $\alpha\geqslant\frac{\gamma_1+2}{\gamma_1-2}$&$\mathcal{O}\left(n^{-\frac{2\gamma_2}{\gamma_2-2}}\right)$\tabularnewline
  \hline
\end{tabularx}
\caption{Convergence rate of $F(x_n)-F^*$ for FISTA under geometry assumptions on $F$.}\label{tab:fista_nu}
\end{table}}

The convergence results stated above give strong guarantees but they all rely on the hypothesis that $F$ has a unique minimizer. Similarly, this assumption appears in \cite{Tao2016local} when proving the linear convergence of FISTA iterates for a LASSO problem. We can observe that in each aforementioned case, this condition allows to prove trivially the strong convergence of FISTA iterates towards the unique minimizer $x^*$ of $F$: we know that $\lim\limits_{n\rightarrow+\infty}F(x_n)-F^*=0$ and $K\|x_n-x^*\|^\gamma\leqslant F(x_n)-F^*$ for some $\gamma\geqslant2$ due to the considered growth assumption. Hence, $\|x_n-x^*\|\rightarrow0$ when $n\rightarrow+\infty$.

\subsection{The Asymptotic Vanishing Damping (AVD) system}
In the seminal work by Su et al. \cite{su2014differential}, the authors demonstrate that the Fast Iterative Shrinkage-Thresholding Algorithm (FISTA), within a differentiable framework, can be interpreted as the discretization of the following ordinary differential equation (ODE) called Asymptotic Vanishing Damping (AVD) system \cite{nesterov1983method,su2014differential}

\begin{equation}\label{eq:Nesterov_ODE}\tag{AVD}
\ddot{x}(t) + \frac{\alpha}{t} \dot{x}(t) + \nabla F(x(t)) = 0,
\end{equation}
where $\alpha = 3$.

The connection between inertial algorithms and ODEs dates back to the pioneering work of Polyak \cite{polyak1964some} on Heavy Ball schemes. In Polyak's observations, the following equation describes the evolution of a particle subject to a force field described by $\nabla F$ and a potentially time-dependent friction term $\alpha(t)$:

\begin{equation}\label{ODEHB}
\ddot{x}(t) + \alpha(t) \dot{x}(t) + \nabla F(x(t)) = 0.
\end{equation}

If $F$ is $\mu$-strongly convex, Polyak demonstrates that the optimal friction is constant, depending on $\mu$, ensuring an exponential decay of $F(x(t)) - F^*$.

Attouch et al. \cite{attouch2017asymptotic,attouch2018fast} provide a comprehensive study of the solution to the ODE \eqref{ODEHB} based on the properties of $F$ and the friction $\alpha(t)$, in particular analyzing the ODE \eqref{eq:Nesterov_ODE}. In both papers, the authors provide convergence rates for $F(x(t)) - F^*$ in the strongly convex case and for functions growing quadratically with a unique minimizer. Specifically, they show that:

\begin{equation}
F(x(t)) - F^* = O\left(t^{-\frac{2\alpha}{3}}\right).
\end{equation}

Aujol et al. \cite{aujol2019optimal} demonstrate that these convergence rates can be improved by introducing an assumption of flatness, also known as quasar convexity. Under weaker growth conditions and quasar convexity, Aujol et al. \cite{aujol2019optimal} and later Luo et al. \cite{Luo2023} provide new convergence rates for the solution of \eqref{eq:Nesterov_ODE}.

All these results assume that the function $F$ to be minimized admits a unique minimizer. These findings are summarized in Table \ref{tab:avd_nu}.

Weak convergence of $x(\cdot)$ towards a minimizer of $F$ has been demonstrated by Attouch et al. \cite{attouch2018fast} by adapting the convergence strategy proposed for the iterates of FISTA by Chambolle et al. \cite{chambolle2015convergence}. Under the assumption of convexity of $F$, strong convergence is straightforward if $F$ is strongly convex or if $F$ grows quadratically with a unique minimizer, but less clear without these assumptions. In their work, the authors propose several sets of assumptions, such as the parity of $F$ or the non-emptiness of the interior of the set of minimizers of $F$, to ensure strong convergence of $x(\cdot)$ towards a minimizer $x^*$ of $F$.

In Section \ref{sec:AVD}, we present new results on convergence rates under growth assumptions without assuming uniqueness of the minimizer. The strong convergence of the trajectory towards a minimizer of $F$ is also proved by showing its finite length.

\off{
Su et al.

Results under geometry assumptions

We consider the ODE associated to Nesterov's accelerated gradient scheme which we call Asymptotic Vanishing Damping (AVD) system \cite{nesterov1983method,su2014differential}:
\begin{equation}
\forall t\geqslant t_0,\quad\ddot{x}(t)+\frac{\alpha}{t}\dot{x}(t)+\nabla F(x(t))=0,
\label{eq:Nesterov_ODE}
\tag{AVD}
\end{equation}
where $t_0>0$ and $\alpha>0$.

The convergence of the trajectories of this ODE was analysed in plenty of works. We focus here on functions satisfying some geometry assumption which was done in \cite{aujol2019optimal,sebbouh2020convergence,aujol2023fista}. The following theorems provide an extension of the convergence results obtained in the aforementioned works. These results require an assumption on the geometry of the set of minimizers $X^*$.
}

{\renewcommand{\arraystretch}{1.5}
\begin{table}[H]
\centering
 \begin{tabularx}{\linewidth}{|>{\centering}X|>{\centering}X|>{\centering}X|>{\centering}X|}
  \hline
  \textbf{Reference} & \textbf{Assumption on $F$} & \textbf{Parameter range} & \textbf{Convergence rate of $F(x(t))-F^*$} \tabularnewline
  \hline
  Su et al. \cite{su2014differential}& $\mathcal{S}_\mu$ & $\alpha\geqslant \frac{9}{2}$ & $\mathcal{O}\left(t^{-3}\right) $\tabularnewline
  \hline
  Attouch et al. \cite{attouch2018fast} & $\mathcal{S}_\mu$ & $\alpha>3$ & $\mathcal{O}\left(t^{-\frac{2\alpha}{3}}\right) $
\tabularnewline
  \hline
  Aujol et al. \cite{aujol2019optimal,aujol2023fista}  & $\mathcal{F}_\gamma$ and $\mathcal{G}^2_\mu$\\Uniqueness of the minimizer & $\alpha>1+\frac{2}{\gamma}$ & $\mathcal{O}\left(t^{-\frac{2\alpha\gamma}{\gamma+2}}\right)$ \tabularnewline
  \hline
  Aujol et al. \cite{aujol2019optimal} & $\mathcal{F}_{\gamma_1}$ and $\mathcal{G}^{\gamma_2}$ where $ \gamma_2\geqslant\gamma_1>2$\\$F$ coercive & $\alpha\geqslant \frac{\gamma_1+2}{\gamma_1-2}$ & $\mathcal{O}\left(t^{-\frac{2\gamma_2}{\gamma_2-2}}\right)$ \tabularnewline\hline
  Luo, Xiao \cite{Luo2023} & $\mathcal{F}_{\gamma_1}$ and $\mathcal{G}^{\gamma_2}$ where $ \gamma_2\geqslant\gamma_1>2$\\Uniqueness of the minimizer & $\alpha\in\left(\frac{\gamma_1+2}{\gamma_1},\frac{\gamma_1+2}{\gamma_1}\cdot\frac{\gamma_2}{\gamma_2-2}\right)$ & $\mathcal{O}\left(t^{-\frac{2\alpha\gamma_1}{\gamma_1+2}}\right)$ \tabularnewline\hline
\end{tabularx}
\caption{Convergence rate of $F(x(t))-F^*$ where $x$ is solution of \eqref{eq:Nesterov_ODE} under geometry assumptions on $F$.}\label{tab:avd_nu}
\end{table}}

\section{Strong convergence of FISTA iterates}\label{sec:Holderian}

In this section, we establish the \textbf{strong convergence of FISTA iterates to a minimizer} of a composite function $F\in\mathcal{C}$ (see Definition \ref{def:Composite}) if this function has a H\"olderian or quadratic growth. Recall that iterates of FISTA are defined as:
\begin{equation}\label{eq:FISTA}
    x_0\in\mathcal{H},\quad\forall n\in\N,\left\{\begin{aligned}
        &y_n=x_n+\alpha_n\left(x_n-x_{n-1}\right)\\
        &x_{n+1}=\prox_{sh}\left(y_n-s\nabla f(y_n)\right),
    \end{aligned}\right.
\end{equation}
where $x_{-1}=x_0$ and we choose the Chambolle-Dossal definition of $\left(\alpha_n\right)_{n\in\N}$ i.e. $\alpha_n=\frac{n}{n+\alpha}$ with $\alpha>3$.

This property stated in Theorem \ref{thm:FISTA_flat}, Corollary \ref{cor:strong_cvg_flat} and Theorem \ref{thm:strong_cvg} relies on asymptotic controls of the sequence $\left(\|x_n-x_{n-1}\|\right)_{n\in\N}$ ensuring that the trajectory described by FISTA iterates has a finite length. Worst-case convergence rates for the error are given based on Lyapunov analyses and using the links between FISTA and \eqref{eq:avd_intro}. We also provide convergence guarantees in the continuous setting under similar assumptions in Section \ref{sec:AVD}.

\subsection{H\"olderian growth condition}

We first consider functions satisfying the local H\"olderian growth condition $\mathcal{G}_{loc}^\gamma$ for $\gamma>2$ and give convergence rates for FISTA iterates.
\begin{theoreme}
\label{thm:FISTA_flat}
Let $F\in \mathcal{C}$ be a coercive composite function
having a H\"olderian growth i.e. satisfying $\mathcal G^\gamma_\text{loc}$ for some $\gamma>2$. Then for $\alpha>5+\frac{8}{\gamma-2}$, the sequence $(x_n)_{n\in\N}$ provided by \eqref{eq:FISTA} with $s=\frac{1}{L}$ satisfies:
\begin{equation}\label{eq:rate_flat}
F(x_n)-F^*=\mathcal{O}\left(n^{-\frac{2\gamma}{\gamma-2}}\right),\quad \|x_n-x_{n-1}\|=\mathcal{O}\left(n^{-\frac{\gamma}{\gamma-2}}\right).
\end{equation}
Moreover the trajectory $(x_n)_{n}$ has a finite length and strongly converges to a minimizer $x^*$ of $F$. 
\end{theoreme}
The proof of Theorem~\ref{thm:FISTA_flat} is detailed in Section \ref{sec:proof_fistaflat}. Note that this theorem can be seen as a discrete version of Theorem \ref{thm:AVD_flat} giving properties of the solution of the ODE associated to Nesterov and presented in Section~\ref{sec:AVD}. 

Several comments can be made about Theorem~\ref{thm:FISTA_flat}. First note that the strong convergence of the sequence $(x_n)_{n\in\N}$ is a consequence of the summability of $(\|x_n-x_{n-1}\|)_{n\in\N}$. Also observe that the convergence rate \eqref{eq:rate_flat} is faster than the one achieved by the Proximal Gradient descend, see Section \ref{Sec:ConvGeo} for more details. Hence, FISTA provides an improvement for the class of convex functions satisfying a local H\"olderian growth condition. Similar bounds have been established by Apidopoulos et al. \cite{apidopoulos2021convergence} but the assumptions of Theorem \ref{thm:FISTA_flat} are weaker: no flatness hypothesis and no uniqueness of the minimizer are required. 


Lastly, the conclusions of Theorem \ref{thm:FISTA_flat} hold if the composite function $F$ satisfies $\mathcal G^\gamma_\text{loc}$ for $\gamma>2$ and for $\alpha>5+\frac{8}{\gamma-2}$. Remarking that $F$ satisfies $\mathcal G^{\gamma'}_\text{loc}$ for any $\gamma'\geqslant \gamma$ 
and that Theorem~\ref{thm:FISTA_flat} thus holds for any $\gamma'>\max (\gamma, \frac{8}{\alpha-5}+2)$, we deduce the following Corollary :

\begin{corollaire}\label{cor:strong_cvg_flat}
    Let $F\in\mathcal{C}$ be a coercive composite function having a H\"olderian growth i.e. satisfying $\mathcal G^\gamma_\text{loc}$ for $\gamma>2$. Then, for any $\alpha>5$, {the sequence $(x_n)_{n\in\N}$ provided by \eqref{eq:FISTA} converges strongly to a minimizer of $F$}.
\end{corollaire}

Finally, observe that the growth properties required in Theorem~\ref{thm:FISTA_flat} are only local and thus, the decays are asymptotic. Even if the proof of Theorem~\ref{thm:FISTA_flat} relies on a Lyapunov analysis, it seems technically difficult in this H\"olderian setting to exhibit explicit bounds for a given number of iteration $n$. 


\subsection{Quadratic growth condition}

In this section, we consider that $F$ has a quadratic growth (denoted by $\mathcal{G}^2_\mu$ for the global growth condition and $\mathcal G_{\mu,loc}^2$ for the local one) with parameter $\mu>0$. This assumption is more restrictive than the Hölderian growth condition considered in Section \ref{sec:Holderian}, and allows to derive stronger convergence results. 

\begin{theoreme}
\label{thm:FISTA}
Let $F\in\mathcal{C}$ be a composite coercive function satisfying a quadratic growth condition $\mathcal G_\mu^2$ for some real parameter $\mu>0$. 
Let $\alpha \geqslant 3+\frac{3}{\sqrt{2}}$ and $\kappa=\frac{\mu}{L}$. Then there exist $\kappa_0>0$ such that for any $0<\kappa \leqslant \kappa_0$, the sequence $(x_n)_{n\in\N}$ generated by FISTA with $s=\frac{1}{L}$ satisfies:
 \begin{equation}\label{bornegene}
\forall n\geqslant \frac{3\alpha}{\sqrt{\kappa}},~F(x_n)-F^*\leqslant\frac{9}{4}e^{-2}M_0\left(\frac{8e}{3\sqrt{\kappa}}\alpha\right)^\frac{2\alpha}{3}n^{-\frac{2\alpha}{3}},
\end{equation}
where $M_{0}= F(x_0)-F^*$ denotes the potential energy of the system at initial time.
\end{theoreme}
Theorem~\ref{thm:FISTA}, whose proof is detailed in Section \ref{sec:proof_FISTA}, is an extension of \cite[Theorem~6]{aujol2023fista} to the class of composite functions with a set of minimizers not reduced to a single point. Similar results can be demonstrated by assuming that $F$ is coercive and only satisfies some local quadratic growth condition. Indeed, the worst-case convergence rate of FISTA \eqref{eq:cv_FISTA_cv} is well known (see \cite{su2014differential,chambolle2015convergence}) and in particular, we know that the sequence $(F(x_n)-F^*)_{n\in \N}$ converges to $0$. Then, according to Lemma \ref{lem_cv_locale}, so does the distance $\left(d(x_n,X^*)\right)_{n\in\N}$ of the iterates to the set of minimizers. Thus, all the inequalities used and demonstrated in the proof of Theorem~\ref{thm:FISTA} remain valid for $n$ large enough and the obtained convergence rates thus hold asymptotically. Our main contribution is to show that under local quadratic growth assumption and without minimizer uniqueness assumption, the trajectory of FISTA iterates is of finite length and strongly converges to a minimizer of $F$:
\begin{theoreme}\label{thm:strong_cvg}
Let $F\in\mathcal{C}$ be a composite coercive function satisfying a local quadratic growth condition $\mathcal G_{\mu,loc}^2$ for some real parameter $\mu>0$. Then for any $\alpha\geqslant3+\frac{3}{\sqrt{2}}$, the sequence $(x_n)_{n\in\N}$ of iterates provided by \eqref{eq:FISTA} with $s=\frac{1}{L}$, satisfies:
\begin{equation}
F(x_n)-F^*=\mathcal{O}\left(n^{-\frac{2\alpha}{3}}\right),\quad \|x_n-x_{n-1}\|=\mathcal{O}\left(n^{-\frac{\alpha}{3}}\right).\label{asymptotic:rate}
\end{equation}
Moreover the trajectory $(x_n)_{n}$ has a finite length and strongly converges to a minimizer $x^*$ of $F$. 
\end{theoreme}

Thus, under the quadratic growth property, we find the rate of convergence in $\mathcal{O}\left(n^{-\frac{2\alpha}{3}}\right)$  known until now only for FISTA under uniqueness of the minimizer. Moreover, observe that if the quadratic growth hypothesis is assumed to be global, Theorem \ref{thm:FISTA} provides explicit non-asymptotic bounds that can be used to parameterize FISTA as it was done in \cite{aujol2023fista}.

More precisely, let $\varepsilon>0$. The minimizers of the composite function $F$ can be characterized by the optimality condition $0\in \partial F(x)$, or equivalently $g(x)=0$ where:
\begin{equation}
    g(x)=L(x-x^+) :=L\left(x- \prox_{\frac{1}{L}h}(x-\frac{1}{L} \nabla f(x))\right),~x\in \mathcal{H},
\end{equation}
denotes the composite gradient mapping and $x^+ :=\prox_{\frac{1}{L}h}(x-\frac{1}{L} \nabla f(x))$. This last formulation is convenient for defining an approximate solution to the composite problem, and thus to deduce a tractable stopping criterion:
\begin{definition}[$\varepsilon$-solution]
Let $\varepsilon$ be the expected accuracy. The iterate $x_n$ is said to be an $\varepsilon$-solution of the problem
$\min_{x\in \mathcal{H}} F(x)$ if:
\begin{equation}
\|g(x_n)\| \leqslant \varepsilon.\label{stop}
\end{equation}
\end{definition}
Observe that in the differentiable case (i.e. when $h=0$), we have: $g(x)=\nabla f(x)$ so that an $\varepsilon$-solution is nothing more than an iterate $x_n$ satisfying:
\begin{equation}
\|g(x_n)\|=\|\nabla F(x_n)\| \leqslant \varepsilon.
\end{equation}
The notion of $\varepsilon$-solution can be seen as a good stopping criterion for an algorithm solving the composite optimization problem for the following reasons. It is numerically quantifiable and in addition, controlling the norm of the composite gradient mapping is roughly equivalent to having a control on the values of the objective function. Indeed using \cite[Theorem 1]{nesterov2007gradient} and \cite[Lemma 3.1]{aujol2023fista}, we can prove that the composite gradient mapping is controlled by the values of the objective function:
\begin{equation}
\forall x\in \R^N,~\frac{1}{2L}\|g(x)\|^2\leqslant F(x) -F^*.\label{control1}
\end{equation}
Hence, from Theorem~\ref{thm:FISTA}, a sufficient condition to reach an $\varepsilon$-solution is:
\begin{equation}
    \frac{9L}{2}e^{-2}M_0\left(\frac{8e}{3\sqrt{\kappa}}\alpha\right)^\frac{2\alpha}{3}n^{-\frac{2\alpha}{3}}
    \leqslant \varepsilon^2,
\end{equation}
which amounts to 
\begin{equation}
    n \geqslant \left( \sqrt{\frac{LM_0}{2}} \frac{3}{e \, \varepsilon}
    \right)^\frac{3}{\alpha}
  \frac{8e}{3\sqrt{\kappa}}\alpha.
\end{equation}
Minimizing the number of iterations to reach an $\varepsilon$-solution with respect to the friction parameter $\alpha$,
%
we thus deduce that choosing
\begin{equation}
    \alpha = \alpha_\varepsilon := 3 \log
  \left(
  \frac{3}{e \, \varepsilon}\sqrt{\frac{LM_0}{2}} 
  \right),
    \end{equation}
will ensure to reach an $\varepsilon$-solution in at most:
\begin{equation}
 \ n_{\varepsilon} :=\frac{8e^2}{\sqrt{\kappa}} \log
  \left(
  \frac{3}{e \, \varepsilon}\sqrt{\frac{LM_0}{2}} 
  \right)\label{nbiterationFISTA}
\end{equation}
iterations. In other words, for a fixed precision $\varepsilon>0$, it is possible to parameterize FISTA such that the number of iterations to reach an $\varepsilon$-solution is comparable to the number of iterations required by an algorithm with an exponential decay.

Notice that in the case of FISTA with the assumption of a unique minimizer \cite{aujol2023fista}, for the exact same choice of $\alpha=\alpha_\varepsilon$ (which is not the optimized choice stated in \cite[Theorem 3]{aujol2023fista}), the number of iterations (denoted by $n_\varepsilon^{FISTA,uniq})$ to reach an $\varepsilon$-solution is then:
\begin{eqnarray}
n_{\varepsilon}^{FISTA,uniq} &=&\frac{8e^2}{3\sqrt{\kappa}}\alpha_{\varepsilon} = \frac{8e^2}{\sqrt{\kappa}}\log\left(\frac{5\sqrt{LM_0}}{e\sqrt{2}\varepsilon}\right),
\end{eqnarray} 
which is better than that given by \eqref{nbiterationFISTA} for FISTA without the minimizer uniqueness assumption:
\begin{equation}
    n_{\varepsilon} =  n_{\varepsilon}^{FISTA,uniq} + \frac{8e^2}{\sqrt{\kappa}}\log\left(\frac{3\sqrt{2}}{5}\right) > n_{\varepsilon}^{FISTA,uniq}.
\end{equation}

\begin{remarque}
The convergence rate stated in Theorem~\ref{thm:FISTA} can be strengthened if there exists $\gamma>1$ such that some flatness condition is satisfied:
\begin{equation}
\forall x\in\mathcal{H},~F(x)-F^*\leqslant\frac{1}{\gamma}\langle \nabla F(x),x-x^*\rangle,
\end{equation}
for any minimizer $x^*\in X^*$, as it was done in \cite[Theorem 4]{aujol2023fista}.
\end{remarque}

\section{Asymptotic Vanishing Damping system under geometry conditions}\label{sec:AVD}
Let us now consider the \ref{AVD} system 
\begin{equation}\tag{AVD}\label{AVD}
\ddot{x}(t) + \frac{\alpha}{t} \dot{x}(t) + \nabla F(x(t)) = 0,
\end{equation}
which has been widely studied in the literature, in particular using Lyapunov-type approaches (see e.g. Table \ref{tab:avd_nu} for a short overview). Let us mention the references \cite{su2014differential, attouch2018fast, aujol2019optimal,aujol2023fista} that introduce the following energy:
\begin{equation}
\mathcal E(t) = t^2\left(F(x(t))-F^*\right) +\frac{1}{2}\|\lambda(x(t)-x^*)+t\dot{x}(t)\|^2\label{lyapunov:continu}
\end{equation}
with different values of $\lambda > 0$, depending on a given minimizer $x^*$ which is supposed to be constant in time. The uniqueness assumption of the minimizer is not necessary to obtain the results proved by Su, Boyd and Candès \cite{su2014differential} and Attouch, Chbani, Peypouquet and Redont \cite{attouch2018fast} in the convex case. On the other hand, when assuming an additional growth property, the fact that these energies depend on a fixed $x^*\in X^*$ is limiting for determining improved convergence rates. Our approach to extend classical analysis without the uniqueness assumption (similar to that in \cite{aujol2024heavyballmomentumnonstrongly}) consists in slightly modifying the Lyapunov energy \eqref{lyapunov:continu} as follows:
\begin{equation}
\mathcal E(t) = t^2\left(F(x(t))-F^*\right) +\frac{1}{2}\|\lambda(x(t)-x^*(t))+t\dot{x}(t)\|^2+\frac{\xi}{2}\|x(t) -x^*(t)\|^2\label{lyapunov:continu2}
\end{equation}
where $x^*(t)$ denotes the projection of the trajectory $x(t)$ onto the set of minimizers $X^*$:
$$x^*(t)=P_{X^*}(x(t)):=\arg\min\limits_{x^*\in X^*}\|x(t)-x^*\|^2.$$
Note that since $F$ is assumed to be continuous and convex, the set $X^*$ is actually a closed convex set and the projection onto $X^*$ is thus well defined. This modification of the energy $\mathcal E$ leads to a question when attempting to conduct the Lyapunov analysis: is $t\mapsto x^*(t)$ differentiable? 

The smoothness of $t\mapsto x^*(t)$ is related to the smoothness of $P_{X^*}$. In fact, if $P_{X^*}$ is directionally differentiable then $t\mapsto x^*(t)$ is right-differentiable (and left-differentiable) and its right-hand derivative is equal to $P^\prime_{X^*}(x(t),\dot{x}(t))$. We refer the reader to Appendix \ref{sec:discussion_NU} for more insightful explanations.

In \cite[Theorem~7.2]{bonnans1998sensitivity}, Bonnans et al. prove that if a closed convex set $\mathcal{S}\subset\mathcal{X}$ is second order regular at $P_\mathcal{S}(x)$ for some $x\in\mathcal{X}$, then $P_\mathcal{S}$ is directionally differentiable at $x$. 
\begin{definition}{\cite[Definition 2.1]{shapiro2016differentiability}}\label{def:second_order_regular}
A set $S$ is said second order regular at a point $\bar x\in S$ if for any sequence $(x_n)_{n\in \N}$ in $S$ of the form: $x_n=\bar x+t_nh +\frac{1}{2}t_n^2r_n$, where $(t_n)_{n\in \N}$ is monotonically non-increasing to $0$, $t_nr_n\rightarrow 0$ and $h\in \mathcal H$, it follows that:
$$\lim_{n\rightarrow +\infty} d(r_n,T_S^2(\bar x,h))=0,$$
where $T_S^2(\bar x,h)$ denotes the inner second order tangent set to $S$ in the direction $h$:
$$T_S^2(\bar x,h):=\left\{w\in \mathcal H~:~d(\bar x+th +\frac{1}{2} t^2w,S) = o\left(t^2\right) \right\}.$$
The set $S$ is said second order regular if it is second order regular at every point.
\end{definition}

We refer the reader to \cite{bonnans1998sensitivity,shapiro2016differentiability} to have a complete understanding of the complex notion of second order regularity. Keep in mind that sets having a $C^2$ boundary \cite{hiriart1982points} (in the sense that their boundary is locally a $C^2$ sub-manifold of $\mathcal H$) and polyhedral sets \cite{shapiro2016differentiability} are second-order regular, so that the projection onto these sets is actually directionally differentiable.

Assuming that the set of minimizers $X^*$ is second order regular instead of the classical uniqueness assumption, Theorem \ref{thm:AVD_flat} provides new bounds on $F(x(t))-F(x^*)$ and on $\norm{\dot x(t)}$ under Hölderian growth conditions. The proof is detailed in Appendix \ref{sec:proof_flat}.


\begin{theoreme}
\label{thm:AVD_flat}
Let $F$ be a convex differentiable function with a non-empty second order regular set of minimizers $X^*$. If $F$ is coercive and satisfies a Hölderian growth condition $\mathcal{G}_{loc}^\gamma$ for some $\gamma>2$. Then, for any $\alpha>\frac{9}{2}+\frac{6}{\gamma-2}$, the trajectories provided by \eqref{eq:Nesterov_ODE} satisfy
\begin{equation}
F(x(t))-F^*=\mathcal{O}\left(t^{-\frac{2\gamma}{\gamma-2}}\right),~\|\dot{x}(t)\|=\mathcal{O}\left(t^{-\frac{\gamma}{\gamma-2}}\right),
\end{equation}
and strongly converge to a minimizer of $F$.
\end{theoreme}
Unlike Aujol et al. \cite{aujol2019optimal} and Luo et al. \cite{Luo2023}, no flatness condition on $F$ or uniqueness of the minimizer is needed here. The only added hypothesis is the regularity of the set of minimizers. This hypothesis may be technical, but seems difficult to remove.
Note that the bound on $\|\dot{x}(\cdot)\|$ implies that the trajectory $x(\cdot)$ has a finite length and strongly converges to a minimizer of $F$. 

Finally, we consider the class of convex differentiable functions having a quadratic growth. Applying the strategy described at the beginning of this section and in Appendix \ref{sec:discussion_NU}, we propose an extension of \cite[Theorem~5]{aujol2023fista} to functions having a set of minimizers not reduced to a single point, and complement this theorem with a result on $\|\dot x(\cdot)\|$ ensuring that the trajectory $x(\cdot)$ has finite length and thus strongly converges to a minimizer $x^*$ of $F$.

\begin{theoreme}\label{thm:continu1}
Let $F$ be a convex differentiable function with a non-empty second order regular set of minimizers $X^*$. 
Assume that $F$ is coercive and satisfies a local quadratic growth condition $\mathcal{G}^2_{\mu,loc}$ for some $\mu>0$. Let $x$ be a solution of \eqref{eq:Nesterov_ODE} for some $t_0\geqslant0$ and $\alpha>0$. If $\alpha>3$ and $\mu$ is small enough then we have:
\begin{equation}
F(x(t)) - F^* =\mathcal O\left(t^{-\frac{2\alpha}{3}}\right),~\|\dot{x}(t)\|=\mathcal{O}\left(t^{-\frac{\alpha}{3}}\right).
\end{equation}
and the trajectory $x(\cdot)$ strongly converges to a minimizer of $F$.
\end{theoreme}
Note that this rate in $\mathcal O\left(t^{-\frac{2\alpha}{3}}\right)$ was already known but for classes of functions satisfying stronger geometric assumptions, in particular for strongly convex functions in \cite[Theorem 8]{su2014differential} and for convex functions having a strong minimizer in \cite[Theorem 3.12]{attouch2017asymptotic}. 

Assuming now that $F$ satisfies a global quadratic growth hypothesis, explicit bounds on the decay of the functional can be calculated. This will subsequently allow for an optimized choice of friction parameter values $\alpha$:
\begin{proposition}\label{thm:continu2}
Let $F$ be a convex differentiable function with a non-empty second order regular set of minimizers $X^*$. 
Assume that $F$ satisfies a global quadratic growth condition $\mathcal{G}^2_{\mu}$ for some $\mu>0$. Let $x$ be a solution of \eqref{eq:Nesterov_ODE} for some $t_0\geqslant0$ and $\alpha>0$. If $\alpha>3$ and $\mu$ is small enough then we have:
\begin{equation}
\forall t\geqslant \frac{\alpha r^*}{3\sqrt{\mu}}\geqslant t_0,~F(x(t))-F^*\leqslant C_1e^{\frac{2}{3}C_2(\alpha-3)}M_0\left(\frac{\alpha r^*}{3t\sqrt{\mu}}\right)^{\frac{2\alpha}{3}},\label{eq:avd_rate}
\end{equation}\normalsize
where $M_0=F(x(t_0))-F^*+\frac{1}{2}\|\dot{x}(t_0)\|^2$, $r^*\simeq 3$ is the unique positive real root of the polynomial: $r\mapsto r^3-r^2-2(1+\sqrt{2})r-4$ and
\begin{eqnarray*}
C_1 = 1+\frac{2}{r^{*}}+\frac{4}{r^{*2}},~
C_2 =\frac{1}{r^*} + \frac{1+\sqrt{2}}{r^{*2}}+\frac{4}{3r^{*3}}.
\end{eqnarray*}
\end{proposition}
We give a simplified analysis of this bound by removing some of the constants in the bound \eqref{eq:avd_rate} for more readability. Let $\varepsilon>0$ be the desired precision on the functional decay $F(x(t))-F^*$. For any $\alpha>3$, the minimum time $t$ to reach the precision $\varepsilon$ is at least in:
 \begin{eqnarray*}
 \left(\frac{\alpha}{t\sqrt{\mu}}\right)^\frac{2\alpha}{3} \leqslant \varepsilon &\Longleftrightarrow & t \geqslant \frac{\alpha}{\sqrt{\mu}} \left(\frac{1}{\varepsilon}\right)^\frac{3}{2\alpha}
\end{eqnarray*}
which corresponds to the polynomial rate stated in Theorem \ref{thm:continu1}. Choosing now $\alpha = C \log\left(\frac{1}{\varepsilon}\right)$ for a well-chosen real constant $C>0$, the minimum time $t$ to reach an $\varepsilon$-solution is at least in:
\begin{eqnarray*}
 \left(\frac{\alpha}{t\sqrt{\mu}}\right)^\frac{2\alpha}{3} \leqslant \varepsilon &\Longleftrightarrow & t \geqslant \frac{Ce^\frac{3}{2C}}{\sqrt{\mu}} \log\left(\frac{1}{\varepsilon}\right)
\end{eqnarray*}
which is comparable to a fast exponential decay of the trajectory.

\section{Proofs of Theorem \ref{thm:FISTA_flat} and Theorem \ref{thm:FISTA}}\label{sec:proofs_main}


The proofs of Theorems \ref{thm:FISTA_flat},  \ref{thm:FISTA} and  \ref{thm:strong_cvg} are based on a Lyapunov analysis involving similar terms. In particular, the convergence proofs of Theorems \ref{thm:FISTA} and \ref{thm:strong_cvg} are built around
\begin{equation}
E_n=\frac{2n^2}{L}(F(x_n)-F^*)+\left\|\lambda(x_{n-1}-x_{n-1}^*)+n(x_n-x_{n-1})\right\|^2,
\label{eq:Lyap_FISTA}
\end{equation}
where $\lambda>0$, while we consider the following discrete Lyapunov energy for Theorem \ref{thm:FISTA_flat}:
\begin{equation}
\mathcal{E}_n=\frac{2n^2}{L}(F(x_n)-F^*)+\left\|\lambda(x_{n}-x_{n}^*)+n\alpha_n(x_n-x_{n-1})\right\|^2+\xi\|x_n-x_n^*\|^2+\lambda n \alpha_n^2\|x_n-x_{n-1}\|^2,
\end{equation}\normalsize
where $\lambda>0$, $\xi<0$ and $x_n^*$, $n\in \N$, denotes the projection of $x_n$ onto the set of minimizers $X^*$. For the sake of clarity, we introduce the following notations:
\begin{equation}
\begin{gathered}
w_n=\displaystyle\frac{2}{L}(F(x_n)-F^*),~h_n=\|x_n-x^*_n\|^2,~\delta_n=\|x_n-x_{n-1}\|^2,\\
\gamma_n^*=\|x_n^*-x_{n-1}^*\|^2,~\alpha_n=\displaystyle\frac{n}{n+\alpha}.
\end{gathered}\label{eq:notations}
\end{equation}
Both convergence proofs rely on two technical lemma. The first one, whose proof is given in Section \ref{sec:proof_lemma_tech1}, is crucial for handling the non-uniqueness of the minimizer:
\begin{lemme}
\label{lem:tech1}
For all $n\in\N^*$, the following equalities hold:
\begin{enumerate}
	\item $\langle x_n-x_n^*,x_n-x_{n-1}\rangle=\frac{1}{2}(h_n-h_{n-1}+\delta_n-\gamma_n^*)+\langle x_{n-1}-x_{n-1}^*,x_n^*-x_{n-1}^*\rangle.$
	\item $\langle x_{n-1}-x_{n-1}^*,x_n-x_{n-1}\rangle=\frac{1}{2}(h_n-h_{n-1}-\delta_n+\gamma_n^*)+\langle x_{n}-x_{n}^*,x_n^*-x_{n-1}^*\rangle,$
\end{enumerate}
\end{lemme}
The second one encodes the fact that the sequence $(x_n)_{n\in\N}$ is provided by \eqref{eq:FISTA}. Its proof is based on a descent lemma proved in \cite{chambolle2015convergence} and is detailed in Section \ref{sec:proof_lemma_tech2}.
\begin{lemme}
\label{lem:tech2}
Let $(x_n)_{n\in\N}$ be the sequence provided by \eqref{eq:FISTA} with $s=\frac{1}{L}$. Then, for any $n\in\N^*$,
\begin{equation}
w_{n+1}-w_n\leqslant \alpha_n^2\delta_n-\delta_{n+1},
\end{equation}
and
\begin{equation}\begin{aligned}
w_{n+1}&\leqslant (1+\alpha_n)h_n+(\alpha_n^2+\alpha_n)\delta_n-\alpha_n h_{n-1}-h_{n+1}-\gamma_{n+1}^*-\alpha_n\gamma_n^*\\
&+2\alpha_n\langle x_{n-1}-x_{n-1}^*,x_n^*-x_{n-1}^*\rangle-2\langle x_{n+1}-x_{n+1}^*,x_{n+1}^*-x_n^*\rangle,
\end{aligned}
\end{equation}
where $\alpha_n=\frac{n}{n+\alpha}$.
\end{lemme}

We would like to point out that several controls can be deduced from the properties of the projection onto a convex. Indeed, if $C$ is a closed convex set such that $C\subset E$, then for any $x\in E$ and $y\in C$,
$$\langle x-p,y-p\rangle\leqslant0,$$
where $p$ denotes the projection of $x$ onto $C$. This property directly guarantees inequalities such as
$$\langle x_{n}-x_{n}^*,x_n^*-x_{n-1}^*\rangle\geqslant0~~\mbox{ and }~\langle x_{n-1}-x_{n-1}^*,x_n^*-x_{n-1}^*\rangle\leqslant0.$$

\subsection{Proof of Theorem \ref{thm:FISTA_flat}}\label{sec:proof_fistaflat}
\subsubsection{Sketch of the proof}

Recall that our analysis relies on the following discrete Lyapunov energy:
\begin{equation}
\mathcal{E}_n=\frac{2n^2}{L}(F(x_n)-F^*)+\left\|\lambda(x_{n}-x_{n}^*)+n\alpha_n(x_n-x_{n-1})\right\|^2+\xi\|x_n-x_n^*\|^2+\lambda n \alpha_n^2\|x_n-x_{n-1}\|^2,
\label{eq:Lyap_FISTA_flat}
\end{equation}
where $\lambda>0$ and $\xi<0$. Given the notations introduced in \eqref{eq:notations}, it can be rewritten:
\begin{equation}
\mathcal{E}_n=n^2w_n+b_n+\xi h_n+\lambda n\alpha_n^2\delta_n,
\end{equation}
where:\begin{equation}
\begin{array}{l}
w_n=\displaystyle\frac{2}{L}(F(x_n)-F^*),~h_n=\|x_n-x^*_n\|^2,~\delta_n=\|x_n-x_{n-1}\|^2,\\
\gamma_n^*=\|x_n^*-x_{n-1}^*\|^2,~\alpha_n=\displaystyle\frac{n}{n+\alpha},~b_n=\left\|\lambda(x_{n}-x_{n}^*)+n\alpha_n(x_n-x_{n-1})\right\|^2.
\end{array}
\end{equation}

The strategy underlying this proof is to show that this Lyapunov energy behaves asymptotically as $n^{-\frac{4}{\gamma-2}}$. Note that this does not directly guarantee the desired convergence results since $\xi<0$. The local growth condition $\mathcal{G}^\gamma$ satisfied by $F$ is necessary to reach the conclusion. 

In order to study the asymptotic behavior of $\mathcal{E}_n$, we define $\mathcal{J}_n=n^p\mathcal{E}_n$ where $p=1+\frac{4}{\gamma-2}$. The proof then follows several steps:
\begin{itemize}
\item Using the properties of FISTA and the convexity of $F$, we show that for a well-chosen set of parameters $(\alpha,\lambda,\xi)$ and $n$ sufficiently large:
\begin{equation}
\mathcal{J}_{n+1}-\mathcal{J}_n\leqslant A(n+1)^{p+1}w_{n+1}+B(n+1)^{p-1}h_{n+1},
\end{equation}
for some constants $A<0$ and $B>0$.
\item Given the previous inequality and the growth condition satisfied by $F$, we prove that for $n$ sufficiently large:
\begin{equation}
\mathcal{J}_n\leqslant C n,
\end{equation}
for some constant $C>0$. This inequality ensures that $\mathcal{E}_n$ decreases asymptotically as $n^{-\frac{4}{\gamma-2}}$.
\item By coming back to the definition of $\mathcal{J}_n$ and $\mathcal{E}_n$ and using the assumption $\mathcal{G}^\gamma$ satisfied by $F$, we show that $n^{p+1}w_n$ and $\alpha_n^2n^{p+1}\delta_n$ are bounded which leads to the desired results.
\end{itemize}

\subsubsection{A technical Lemma before the proof of Theorem \ref{thm:FISTA_flat}}
In the proof of Theorem~\ref{thm:FISTA_flat}, the geometry of the function $F$ will be useful to control the distance of the FISTA iterates to the set of minimizers by the decay of $F$ along the trajectory of iterates.
\begin{lemme}
\label{lem:geo}
Let $F$ satisfy $\mathcal{G}^\gamma_\text{loc}$ for some $\gamma>2$ and real constant $K>0$. If $p=1+\frac{4}{\gamma-2}$, then for $n$ sufficiently large,
\begin{equation}
n^{p-1}h_n\leqslant \left(\frac{L}{2K}\right)^\frac{2}{\gamma}\left(n^{p+1}w_n\right)^\frac{2}{\gamma},\label{eq:lemma_g_gamma}
\end{equation}
where: $w_n=\displaystyle\frac{2}{L}(F(x_n)-F^*)$ and $h_n=d(x_n,X^*)^2$.
\end{lemme}
\noindent\textbf{Proof.} %
Assume that $F$ satisfies some local Hölderian growth condition $\mathcal{G}^\gamma_\text{loc}$ for some $\gamma>0$. It is well known (see \cite{su2014differential,chambolle2015convergence}) that the iterates of FISTA with $s=\frac{1}{L}$ and $\alpha\geqslant3$ satisfy the following inequality
\begin{equation*}
\forall n\in\N,~F(x_n)-F^*\leqslant \frac{(\alpha-1)^2L}{2(n+\alpha-2)^2}\|x_0-x^*\|^2,
\end{equation*}
which implies that the sequence $\left(F(x_n)-F^*\right)_{n\in\N}$ converges to $0$. Applying Lemma \ref{lem_cv_locale}, we thus deduce that the sequence $\left(d(x_n,X^*)\right)_{n\in\N}$ converges to $0$ as $n\rightarrow+\infty$ and that there exist $K>0$ and $N\in \N$ such that:
\begin{equation}
    \forall n\geqslant N,~Kd(x_n,X^*)^\gamma\leqslant F(x_n)-F^*.\label{growth:cd}
\end{equation}
or, equivalently:
\begin{equation*}
\forall n\geqslant N,~h_n\leqslant \left(\frac{L}{2K}\right)^\frac{2}{\gamma}w_n^\frac{2}{\gamma}.
\end{equation*}
Choosing $p=1+\frac{4}{\gamma-2}$, the expected inequality \eqref{eq:lemma_g_gamma} holds for any $n\geqslant N$.\hfill\qed\\
\subsubsection{Proof of Theorem~\ref{thm:FISTA_flat}}
Let $(x_n)_{n\in\N}$ be the sequence provided by \eqref{eq:FISTA} and $\left(\mathcal{E}_n\right)_{n\in\N}$ be the Lyapunov energy defined in \eqref{eq:Lyap_FISTA_flat}. The first step of the proof is to get an upper bound on $\mathcal{E}_{n+1}-\mathcal{E}_n$. We provide such an inequality in the following lemma which is proved in Section \ref{sec:proof_tech_flat1}.
\begin{lemme}
\label{lem:ineq_flat1}
Let $\xi=\lambda(\lambda+1-\alpha)$. For any $n\in\N^*$,
\begin{equation}
\begin{aligned}
\mathcal{E}_{n+1}-\mathcal{E}_n\leqslant &\left((2-\lambda)n+1\right)w_{n+1}+B_1(n)b_{n+1}+B_2(n)h_{n+1}+B_3(n)\delta_{n+1}\\
&-B_4(n)\left(\gamma_{n+1}^*-2\langle x_n-x_n^*,x_{n+1}^*-x_n^*\rangle\right),
\end{aligned}
\end{equation}
where :
\begin{itemize}
\item$B_1(n)=\frac{2(\lambda+1-\alpha)}{n+1}+\frac{\alpha(2\lambda+2-\alpha)}{(n+1)^2}$, 
\item$B_2(n)=-\frac{2\lambda^2(\lambda+1-\alpha)}{n+1}+\frac{\alpha\lambda^2(\alpha-2\lambda-2)}{(n+1)^2}$, 
\item$B_3(n)=\alpha(\lambda+2)-(2\lambda^2+2\lambda+1)+\alpha^2\frac{n(\lambda-2)+\lambda-2-2\alpha}{(n+1-\alpha)^2}$ ,
\item$B_4(n)=-2\lambda(\lambda+1-\alpha)-\frac{\alpha^2\lambda}{n+1+\alpha}$.
\end{itemize}
\end{lemme}

We introduce $\mathcal{J}_n=n^p\mathcal{E}_n$ with $p=1+\frac{4}{\gamma-2}$. The next step is to show the following inequality.

\begin{lemme}
\label{lem:ineq_flat2}
Let $\xi=\lambda(\lambda+1-\alpha)$. If $\lambda\leqslant \alpha-1$, then for any $n\in\N^*$,
\begin{equation}
\begin{aligned}
\mathcal{J}_{n+1}-\mathcal{J}_n\leqslant&\left((2-\lambda+p)(n+1)^{p+1}+R_1(n)\right)w_{n+1}\\&+\left((2(\lambda+1-\alpha)+p)(n+1)^{p-1}+R_2(n)\right)b_{n+1}\\&+\left(\lambda(\lambda+1-\alpha)(p-2\lambda)(n+1)^{p-1}+R_3(n)\right)h_{n+1}\\&-n^p B_4(n)\left(\gamma_{n+1}^*-2\langle x_n-x_n^*,x_{n+1}^*-x_n^*\rangle\right),
\end{aligned}
\end{equation}
where:
\begin{eqnarray*}
R_1(n)&=&\left(\lambda-1+p(\lambda-2)\right)n^p+p(\lambda-2)n^{p-1}\\
R_2(n)&=&\left((4\lambda+6+2p-\alpha)\alpha+2\lambda p\right)(n+1)^{p-2}+\alpha^2(6\lambda+8+p)(n+1)^{p-3}\\
&&+2\alpha^3(\lambda+2)(n+1)^{p-4}\\
R_3(n)&=&\left(\lambda^2(\alpha^2+2\alpha+4\lambda p+p+1)+\lambda(\alpha-1)(p-1)\right)(n+1)^{p-2}\\
&&+\lambda^2\alpha\left(2p\lambda+2p+6\lambda\alpha+2\alpha\right)(n+1)^{p-3}+2\alpha^3\lambda^2(\lambda+2)(n+1)^{p-4}.
\end{eqnarray*}
\end{lemme}
The proof is detailed in Section \ref{sec:proof_tech_flat2}. By setting $\lambda=\alpha-1-p$, we get that $\lambda\leqslant \alpha-1$ and:
\begin{equation}
\begin{array}{rcl}
C_1&:=&2-\lambda+p=3+2p-\alpha,\\
C_2&:=&2(\lambda+1-\alpha)+p=-p,\\
C_3&:=&\lambda(\lambda+1-\alpha)(p-2\lambda)=p(\alpha-1-p)(2\alpha-2-3p),
\end{array}
\end{equation}
which implies that for $\alpha>3+2p=5+\frac{8}{\gamma-2}$,
\begin{equation}
C_1<0,\quad C_2<0,\quad C_3>0.
\end{equation}
Considering the order of $R_1(n)$, $R_2(n)$ and $R_3(n)$, this guarantees that for $n$ sufficiently large:
\begin{equation}
\left\{
\begin{gathered}
C_1(n+1)^{p+1}+R_1(n)<\frac{C_1}{2}(n+1)^{p+1},\\
C_2(n+1)^{p-1}+R_2(n)<0,\\
C_3(n+1)^{p-1}+R_3(n)<2C_3(n+1)^{p-1}.
\end{gathered}\right.
\end{equation}
In addition, for the choice $\lambda=\alpha-1-p$ we have that $B_4(n)=2p(\alpha-1-p)-\alpha^2\frac{\alpha-1-p}{n+1+\alpha}$ which is positive for $\alpha>5+\frac{8}{\gamma-2}$ and $n$ sufficiently large. As $\gamma_{n+1}^*\geqslant 0$ and $\langle x_n-x_n^*,x_{n+1}^*-x_n^*\rangle\leqslant0$, this ensures that 
\begin{equation}
n^p B_4(n)\left(\gamma_{n+1}^*-2\langle x_n-x_n^*,x_{n+1}^*-x_n^*\rangle\right)\geqslant 0.
\end{equation}
Hence, if $\alpha>5+\frac{8}{\gamma-2}$, then for $n$ sufficiently large:
\begin{equation}
\mathcal{J}_{n+1}-\mathcal{J}_n\leqslant\frac{C_1}{2}(n+1)^{p+1}w_{n+1}+2C_3(n+1)^{p-1}h_{n+1}.\label{ineq:cle:holder}
\end{equation}
\paragraph{\it $1^{st}$ step: Proving that $F(x_n)-F^*=\mathcal O\left(n^{-\frac{2\gamma}{\gamma-2}}\right)$.} To obtain bounds on the decay of $F$ along the FISTA iterates, we take advantage of the geometry of the function $F$ to minimize. Assuming that $F$ satisfies a local Hölderian growth condition, Lemma~\ref{lem:geo} combined with \eqref{ineq:cle:holder} ensure that for $n$ sufficiently large:
\begin{equation}
\mathcal{J}_{n+1}-\mathcal{J}_n\leqslant \frac{C_1}{2}(n+1)^{p+1}w_{n+1}+2C_3 \left(\frac{L}{2K}\right)^\frac{2}{\gamma}\left((n+1)^{p+1}w_{n+1} \right)^\frac{2}{\gamma},
\end{equation}
which ensures that there exists $M_0\in\R$ such that $\mathcal{J}_{n+1}-\mathcal{J}_n\leqslant M_0$.

Thus, there exists $n_0\in\N$ such that for all $n\geqslant n_0$, $\mathcal{J}_n\leqslant nM_0+M_1$ where $M_1=\mathcal{J}_{n_0}-n_0M_0$. Consequently, we get that for $n$ sufficiently large, $\mathcal{J}_n\leqslant 2nM_0$. Coming back to the definition of $\mathcal{J}$, this implies that:
\begin{equation}
n^{p-1}\left(n^2w_n+b_n+\xi h_n+\lambda n\alpha_n^2\delta_n\right)\leqslant 2M_0.\label{eq:hub_flat}
\end{equation}

Noticing that $\xi=\lambda(\lambda+1-\alpha)<0$, this ensures that for $n$ sufficiently large:
\begin{equation}
n^{p+1}w_n-|\xi|n^{p-1}h_n\leqslant 2M_0,
\end{equation}
and according to Lemma \ref{lem:geo}:
\begin{equation}
n^{p+1}w_n-|\xi|\left(\frac{L}{2K}\right)^\frac{2}{\gamma}\left(n^{p+1}w_n\right)^\frac{2}{\gamma}\leqslant 2M_0.
\end{equation}
The following lemma guarantees that for $n$ sufficiently large, $n^{p+1}w_n$ is bounded.
\begin{lemme}
\label{lem:control_v}
Let $x\in\R^+$, $\delta\in(0,1)$, $K_1>0$ and $K_2>0$. Then,
$$x^\delta(x^{1-\delta}-K_1)\leqslant K_2\quad \implies\quad x\leqslant \left(K_2^{1-\delta}+K_1\right)^{\frac{1}{1-\delta}}.$$
\end{lemme}
As a consequence, there exists $M_2>0$ such that for $n$ sufficiently large, $n^{p+1}w_n\leqslant M_2$ and considering the value of $p$ we have that:
\begin{equation}
F(x_n)-F^*\leqslant \frac{LM_2}{2n^\frac{2\gamma}{\gamma-2}}.
\end{equation}
This proves our first claim: $F(x_n)-F^*=\mathcal O\left(n^{-\frac{2\gamma}{\gamma-2}}\right)$.

\paragraph{\it $2^{nd}$ step: Proving that the trajectory of FISTA iterates has a finite length}
Let us come back to the inequality \eqref{eq:hub_flat} which implies that for $n$ sufficiently large:
\begin{equation}
n^{p-1}\left(n^2w_n+b_n-|\xi|h_n\right)\leqslant 2M_0.\label{eq:traj_ineq}
\end{equation}
By applying the inequality $\|u\|^2\leqslant 2\|u+v\|^2+2\|v\|^2$ to $u=\alpha_n(x_n-x_{n-1})$ and $v=\lambda(x_n-x_n^*)$, we get:
\begin{equation}
b_n\geqslant \frac{n^2\alpha_n^2}{2}\delta_n-\lambda^2h_n.
\end{equation}
Combining this inequality with \eqref{eq:traj_ineq} leads to:
\begin{equation}
\left(n^{p+1}w_n-(\lambda^2+|\xi|)n^{p-1}h_n\right)+\frac{\alpha_n^2}{2}n^{p+1}\delta_n\leqslant 2M_0.
\end{equation}
Then, Lemma \ref{lem:geo} gives us that
\begin{equation}
n^{p+1}w_n-(\lambda^2+|\xi|)n^{p-1}h_n\geqslant n^{p+1}w_n-\left(\frac{L}{2K}\right)^\frac{2}{\gamma}\left(n^{p+1}w_n\right)^\frac{2}{\gamma}.
\end{equation}
The study of the variations of $\varphi:x\mapsto x-\left(\frac{L}{2K}\right)^\frac{2}{\gamma} x^\frac{2}{\gamma}$ shows that there exists a real constant $M_3\in\R$ such that $\varphi$ is bounded from below by $M_3$. Hence for $n$ large enough:
$n^{p+1}w_n-(\lambda^2+|\xi|)n^{p-1}h_n\geqslant M_3,$ and:
\begin{equation}
\delta_n\leqslant \frac{4M_0-2M_3}{\alpha_n^2n^{p+1}},
\end{equation}
and therefore: $\|x_n-x_{n-1}\|=\mathcal{O}\left(n^{-\frac{\gamma}{\gamma-2}}\right)$.

\paragraph{\it $3^{rd}$ step: Proving that the FISTA iterates strongly converge to a minimizer of $F$.} The strong convergence of FISTA iterates can be deduced from the summability of $\|x_n-x_{n-1}\|$ since $\frac{\gamma}{\gamma-2}>1$ for any $\gamma>2$.

\subsection{Proof of Theorem \ref{thm:FISTA}}
\label{sec:proof_FISTA}

The proof of Theorem \ref{thm:FISTA} is an adaptation of the proof of \cite[Theorem~6]{aujol2023fista} without the assumption that $F$ has a unique minimizer. Its structure is similar despite the involvement of additional terms linked to the relaxed setting. The tricky technical aspect is to control these additional terms in order to recover inequalities obtained in the case of uniqueness of the minimizer.

Recall that we consider the discrete Lyapunov energy defined in \eqref{eq:Lyap_FISTA} with the notations \eqref{eq:notations}:
\begin{equation}
E_n=n^2w_n+\lambda^2h_{n-1}+n^2\delta_n+2\lambda n  \langle x_{n-1}-x_{n-1}^*,x_n-x_{n-1}\rangle,\label{lyapunov:th2}
\end{equation}
where $\alpha>3$, $\lambda=\frac{2\alpha}{3}$ and:
\begin{equation}
\begin{array}{l}
w_n=\displaystyle\frac{2}{L}(F(x_n)-F^*),~h_n=\|x_n-x^*_n\|^2,~\delta_n=\|x_n-x_{n-1}\|^2,~\alpha_n=\displaystyle\frac{n}{n+\alpha}.
\end{array}
\end{equation}
Applying the second claim of Lemma \ref{lem:tech1}, the Lyapunov energy \eqref{lyapunov:th2} can be rewritten as:
\begin{equation*}
E_n=n^2 w_n+\lambda n h_n+(\lambda^2-\lambda n)h_{n-1}+(n^2-\lambda n)\delta_n +\lambda n\gamma_n^*+ 2\lambda n\langle x_{n}-x_{n}^*,x_n^*-x_{n-1}^*\rangle.
\end{equation*}
For any $n\in\N^*$, we have:\small
\begin{equation*}
\begin{aligned}
E_{n+1}-\left(1-\frac{\lambda-2}{n}\right)E_n&=(n+1)^2w_{n+1}-\left(1-\frac{\lambda-2}{n}\right)n^2 w_n\\
&+\left((n+1)^2-\lambda(n+1)\right)\delta_{n+1}-\left(1-\frac{\lambda-2}{n}\right)(n^2-\lambda n)\delta_n\\
&+\left(\lambda^2-\lambda(n+1)-\lambda n\left(1-\frac{\lambda-2}{n}\right)\right)h_n+\lambda(n+1)h_{n+1}\\
&-(\lambda^2-\lambda n)\left(1-\frac{\lambda-2}{n}\right)h_{n-1}\\
&+\lambda(n+1)\gamma_{n+1}^*+2\lambda (n+1)\langle x_{n+1}-x_{n+1}^*,x_{n+1}^*-x_{n}^*\rangle\\&-\lambda (n-\lambda+2) \gamma_n^*-2\lambda (n-\lambda+2)\langle x_{n}-x_{n}^*,x_n^*-x_{n-1}^*\rangle.
\end{aligned}
\end{equation*}\normalsize
Elementary computations give that:
\begin{equation*}
(n+1)^2w_{n+1}-\left(1-\frac{\lambda-2}{n}\right)n^2w_n=n\left(n-\lambda+2\right)(w_{n+1}-w_n)+(\lambda n+1)w_{n+1}.
\end{equation*}
Consequently, Lemma \ref{lem:tech2} ensures that for all $n\in\N^*$:\small
\begin{equation*}
\begin{aligned}
&(n+1)^2w_{n+1}-\left(1-\frac{\lambda-2}{n}\right)n^2w_n\\
&\leqslant n\left(n-\lambda+2\right)\left(\alpha_n^2\delta_n-\delta_{n+1}\right)\\
&+(\lambda n+1)\left((1+\alpha_n)h_n+(\alpha_n^2+\alpha_n)\delta_n-\alpha_n h_{n-1}-h_{n+1}-\gamma_{n+1}^*-\alpha_n\gamma_n^*\right)\\
&+2(\lambda n+1)\left(\alpha_n\langle x_{n-1}-x_{n-1}^*,x_n^*-x_{n-1}^*\rangle- \langle x_{n+1}-x_{n+1}^*,x_{n+1}^*-x_n^*\rangle\right).
\end{aligned}
\end{equation*}\normalsize
It follows that:\small
\begin{equation}
\begin{aligned}
E_{n+1}-\left(1-\frac{\lambda-2}{n}\right)E_n&\leqslant A_1(n,\alpha)\delta_n+A_2(n,\alpha)\delta_{n+1}+B_1(n,\alpha)h_{n-1}\\
&+B_2(n,\alpha)h_n+B_3(n,\alpha)h_{n+1}+D_1(n,\alpha)\gamma_{n+1}^*\\
&+D_2(n,\alpha)\gamma_n^*+D_3(n,\alpha)\langle x_{n+1}-x_{n+1}^*,x_{n+1}^*-x_{n}^*\rangle\\
&+D_4(n,\alpha)\langle x_{n}-x_{n}^*,x_n^*-x_{n-1}^*\rangle\\&+D_5(n,\alpha)\langle x_{n-1}-x_{n-1}^*,x_n^*-x_{n-1}^*\rangle,
\end{aligned}
\end{equation}\normalsize
where:
\begin{itemize}\small
\item $A_1(n,\alpha)=\frac{17\alpha^2}{9}-\frac{8\alpha}{3}+2-\alpha\frac{(10\alpha^2-18\alpha+9)n+7\alpha^3-12\alpha^2+6\alpha}{3(n+\alpha)^2}$,
\item $A_2(n,\alpha)=1-\frac{2\alpha}{3}$,
\item $B_1(n,\alpha)=-\frac{2\alpha^2}{9}+\frac{4\alpha}{3}-1+\frac{3\alpha-2\alpha^3}{3(n+\alpha)}+\frac{8\alpha^3-24\alpha^2}{27n}$,
\item $B_2(n,\alpha)=\frac{2\alpha^2}{9}-2\alpha+2-\frac{3\alpha-2\alpha^3}{3(n+\alpha)}$,
\item $B_3(n,\alpha)=\frac{2\alpha}{3}-1$,
\item $D_1(n,\alpha)=\frac{2\alpha}{3}-1$,
\item $D_2(n,\alpha)=-\frac{4\alpha}{3}n-1-\frac{4\alpha}{3}+\frac{10\alpha^2}{9}+\frac{3\alpha-2\alpha^3}{3(n+\alpha)}$,
\item $D_3(n,\alpha)=\frac{4\alpha}{3}-2$,
\item $D_4(n,\alpha)=-\frac{4\alpha}{3}n-\frac{8\alpha}{3}+\frac{8\alpha^2}{9}$,
\item $D_5(n,\alpha)=\frac{4\alpha}{3}n+2-\frac{4\alpha^2}{3} +\frac{\alpha(4\alpha^2-6)}{3(n+\alpha)}$.
\end{itemize}\normalsize
Noticing that $B_3(n,\alpha)=-A_2(n,\alpha)=D_1(n,\alpha)=\frac{1}{2}D_3(n,\alpha)$ and:
\begin{equation*}
B_1(n,\alpha)+B_2(n,\alpha)+B_3(n,\alpha)=\frac{8\alpha^2}{27}\frac{\alpha-3}{n}=\frac{4\alpha K(\alpha)}{3n},
\end{equation*}
where $K(\alpha)=\frac{2\alpha(\alpha-3)}{9}$, we get that
\begin{equation}
\label{eq:checkpoint}
\begin{aligned}
E_{n+1}-\left(1-\frac{\lambda-2}{n}\right)E_n\leqslant&\frac{4\alpha K(\alpha)}{3n}h_n+A_1(n,\alpha)\delta_n+B_1(n,\alpha)(h_{n-1}-h_n)\\
&+B_3(n,\alpha)(h_{n+1}-h_n-\delta_{n+1})+B_3(n,\alpha)\gamma_{n+1}^*\\
&+D_2(n,\alpha)\gamma_n^*+2B_3(n,\alpha)\langle x_{n+1}-x_{n+1}^*,x_{n+1}^*-x_{n}^*\rangle\\
&+D_4(n,\alpha)\langle x_{n}-x_{n}^*,x_n^*-x_{n-1}^*\rangle\\&+D_5(n,\alpha)\langle x_{n-1}-x_{n-1}^*,x_n^*-x_{n-1}^*\rangle.
\end{aligned}
\end{equation}\normalsize

We apply the following technical lemma that is an extension of \cite[Lemma~4]{aujol2023fista}. The proof can be found in Section \ref{sec:proof_lemma_FISTA_AB}.
\begin{lemme}
\label{lem:lemma_FISTA_AB}
Let $n>\lambda$ and $(A,B)\in\R^2$. The following two claims hold:
\begin{enumerate}
    \item 
    \begin{equation}
    \delta_n\leqslant\frac{2}{(n-\lambda)^2}b_n+\frac{8\alpha^2}{9(n-\lambda)^2}h_n,
    \label{eq:ineq_delta_n}
    \end{equation}
    where $b_n = \|\lambda(x_{n-1}-x_{n-1}^*)+n(x_n-x_{n-1})\|^2$ for any $n\in\N^*$.
    \item 
    \begin{equation}\label{eq:AB}
    \begin{aligned}
    A\delta_n+B(h_{n-1}-h_n)\leqslant&\left(2|A+B|+\frac{\sqrt{2}|B|}{\sqrt{\kappa}}\right)\left(1+\frac{4\alpha^2}{9\kappa n^2}\right)\frac{E_n}{(n-\lambda)^2}\\&-B\gamma_n^*+2B\langle x_{n-1}-x_{n-1}^*,x_n^*-x_{n-1}^*\rangle.
    \end{aligned}
    \end{equation}
\end{enumerate}
\end{lemme}
\noindent Inequality \eqref{eq:AB} ensures that for any $n>\lambda$:
\begin{equation*}
\begin{aligned}
\frac{4\alpha K(\alpha)}{3}\frac{h_n}{n}+A_1(n,\alpha)\delta_n&+B_1(n,\alpha)(h_{n-1}-h_n)\leqslant \frac{\hat C_1(n,\alpha,\kappa)E_n}{(n-\lambda)^2}\\&-B_1(n,\alpha)\gamma_n^*+2B_1(n,\alpha)\langle x_{n-1}-x_{n-1}^*,x_n^*-x_{n-1}^*\rangle,
\end{aligned}
\end{equation*}\normalsize
and
\begin{equation*}
\begin{aligned}
B_3(n,\alpha)(h_{n+1}-h_n-\delta_{n+1})\leqslant& \frac{\hat C_2(n,\alpha,\kappa)E_{n+1}}{(n+1-\lambda)^2}-2B_3(n,\alpha)\langle x_{n}-x_{n}^*,x_{n+1}^*-x_{n}^*\rangle\\&+B_3(n,\alpha)\gamma_{n+1}^*,
\end{aligned}
\end{equation*}\normalsize
where 
\begin{equation*}
\hat C_1(n,\alpha,\kappa)=2|\frac{5}{3}\alpha^2-\frac{4\alpha}{3}+1+R(n,\alpha)|+\sqrt{2}\left(\frac{|-\frac{2\alpha^2}{9}+\frac{4\alpha}{3}-1+Q(n,\alpha)|}{\sqrt{\kappa}}\right)\left(1+\frac{4\alpha^2}{9\kappa n^2}\right)+\frac{4\alpha K(\alpha)}{3\kappa n}
\end{equation*}\normalsize
with: 
\begin{equation*}
\begin{gathered}
|R(\alpha,n)|=\left|A_1(n,\alpha)+B_1(n,\alpha)-(\frac{5}{3}\alpha^2-\frac{4\alpha}{3}+1)\right|\leqslant \frac{8\alpha^3}{n}\\ 
|Q(\alpha,n)|=\frac{\alpha^3}{3n}\left|n\frac{3-2\alpha^2}{\alpha^2(n+\alpha)}+8\frac{\alpha-3}{9\alpha}\right| \leqslant \frac{\alpha^3}{n},
\end{gathered}
\end{equation*}
and $\hat C_2(n,\alpha,\kappa)=\left(\frac{2\alpha}{3}-1\right)\left(4+\frac{\sqrt{2}}{\sqrt{\kappa}}\right)\left(1+\frac{4\alpha^2}{9\kappa(n+1)^2}\right)$. Coming back to \eqref{eq:checkpoint}, we get that:
\begin{equation}\label{eq:huge_ineq}
\begin{aligned}
E_{n+1}-\left(1-\frac{\lambda-2}{n}\right)E_n\leqslant& \frac{\hat C_1(n,\alpha,\kappa)E_n}{(n-\lambda)^2}+\frac{\hat C_2(n,\alpha,\kappa)E_{n+1}}{(n+1-\lambda)^2}\\
&+2B_3(n,\alpha)\gamma_{n+1}^*+2B_3(n,\alpha)\langle x_{n+1}-x_{n+1}^*,x_{n+1}^*-x_{n}^*\rangle\\
&-2B_3(n,\alpha)\langle x_{n}-x_{n}^*,x_{n+1}^*-x_{n}^*\rangle+(D_2(n,\alpha)-B_1(n,\alpha))\gamma_n^*\\
&+D_4(n,\alpha)\langle x_{n}-x_{n}^*,x_n^*-x_{n-1}^*\rangle\\&+(D_5(n,\alpha)+2B_1(n,\alpha))\langle x_{n-1}-x_{n-1}^*,x_n^*-x_{n-1}^*\rangle.
\end{aligned}
\end{equation}
Note that for all $n\in\N^*$,
\begin{equation*}
\gamma_n^*+\langle x_n-x_n^*,x_n^*-x_{n-1}^*\rangle-\langle x_{n-1}-x_{n-1}^*,x_n^*-x_{n-1}^*\rangle=\langle x_n-x_{n-1},x_n^*-x_{n-1}^*\rangle,
\end{equation*}
and thus,
\begin{equation*}
2B_3(n,\alpha)\left(\gamma_{n+1}^*+\langle x_{n+1}-x_{n+1}^*,x_{n+1}^*-x_{n}^*\rangle
-\langle x_{n}-x_{n}^*,x_{n+1}^*-x_{n}^*\rangle\right)=2(\lambda-1)\langle x_{n+1}-x_n,x_{n+1}^*-x_{n}^*\rangle.
\end{equation*}
Moreover, we can show that for any $n\geqslant \lambda-2$,
\begin{equation}
D_4(n,\alpha)\leqslant D_2(n,\alpha)-B_1(n,\alpha)\leqslant-(D_5(n,\alpha)+2B_1(n,\alpha))\leqslant -2\lambda\left(n-2\left(\lambda-1\right)\right).
\end{equation}
Since
\begin{equation*}
\gamma_n^*\geqslant0,~\langle x_{n}-x_{n}^*,x_n^*-x_{n-1}^*\rangle\geqslant0,~
\langle x_{n-1}-x_{n-1}^*,x_n^*-x_{n-1}^*\rangle\leqslant0,
\end{equation*}
we get that
\begin{equation*}
\begin{gathered}
(D_2(n,\alpha)-B_1(n,\alpha))\gamma_n^*+D_4(n,\alpha)\langle x_{n}-x_{n}^*,x_n^*-x_{n-1}^*\rangle+(D_5(n,\alpha)+2B_1(n,\alpha))\langle x_{n-1}-x_{n-1}^*,x_n^*-x_{n-1}^*\rangle\\
\leqslant-2\lambda\left(n-2\left(\lambda-1\right)\right)\langle x_n-x_{n-1},x_n^*-x_{n-1}^*\rangle,
\end{gathered}
\end{equation*}
which is negative if $n\geqslant 2\lambda-2$. By taking $n>\max\left\{\lambda,2\lambda-2\right\}=2\lambda-2$ (since $\lambda=\frac{2\alpha}{3}$ and $\alpha>3$), we can combine the above inequality with \eqref{eq:huge_ineq}
\begin{equation}
\begin{aligned}
E_{n+1}-\left(1-\frac{\lambda-2}{n}\right)E_n\leqslant& \frac{\hat C_1(n,\alpha,\kappa)E_n}{(n-\lambda)^2}+\frac{\hat C_2(n,\alpha,\kappa)E_{n+1}}{(n+1-\lambda)^2}\\
&+2(\lambda-1)\langle x_{n+1}-x_n,x_{n+1}^*-x_n^*\rangle.
\end{aligned}
\end{equation}
As $\langle x_{n+1}-x_n,x_{n+1}^*-x_n^*\rangle\leqslant \delta_{n+1}$, for any $n>2\lambda-2$,
\begin{equation*}
E_{n+1}-\left(1-\frac{\lambda-2}{n}\right)E_n\leqslant\frac{\hat C_1(n,\alpha,\kappa)E_n}{(n-\lambda)^2}+\frac{\hat C_2(n,\alpha,\kappa)E_{n+1}}{(n+1-\lambda)^2}+2(\lambda-1)\delta_{n+1}.
\end{equation*}

Then, according to the first claim of Lemma \ref{lem:lemma_FISTA_AB} and the quadratic growth condition that can be rewritten with our notation as $h_n\leqslant\frac{E_n}{\kappa n^2}$ for any $n\in \N$, we get the following:
\begin{equation*}
\delta_{n+1}\leqslant \frac{2}{(n+1-\lambda)^2}b_{n+1}+\frac{8\alpha^2}{9(n+1-\lambda)^2}h_{n+1}\leqslant \frac{2}{(n+1-\lambda)^2}\left(1+\frac{4\alpha^2}{9\kappa(n+1)^2}\right)E_{n+1}.
\end{equation*}
Hence,
\begin{equation}
E_{n+1}-\left(1-\frac{\lambda-2}{n}\right)E_n\leqslant\frac{\tilde C_1(n,\alpha,\kappa)E_n}{(n-\lambda)^2}+\frac{\tilde C_2(n,\alpha,\kappa)E_{n+1}}{(n+1-\lambda)^2},\label{eq:FISTA_ineq_key}
\end{equation}
where $\tilde C_1(n,\alpha,\kappa)=\hat C_1(n,\alpha,\kappa)$ and 
$$\begin{aligned}\tilde C_2(n,\alpha,\kappa)&=\hat C_2(n,\alpha,\kappa)+ 4(\lambda-1)\left(1+\frac{4\alpha^2}{9\kappa(n+1)^2}\right)\\&=\left(\frac{2\alpha}{3}-1\right)\left(8+\frac{\sqrt{2}}{\sqrt{\kappa}}\right)\left(1+\frac{4\alpha^2}{9\kappa(n+1)^2}\right).\end{aligned}$$

As $\kappa\in(0,1]$, for any $n\geqslant \frac{4\alpha}{3\sqrt{\kappa}}$, we have that $\displaystyle\frac{1}{n-\lambda}=\frac{1}{n-\frac{2\alpha}{3}}\leqslant\frac{1}{n}\left(1+\sqrt{\kappa}\right)$ and thus, for any $n\geqslant\frac{4\alpha}{3\sqrt{\kappa}}$,
\begin{equation}
E_{n+1}-\left(1-\frac{\frac{2\alpha}{3}-2}{n}\right) E_n \leqslant (1+\sqrt{\kappa})^2\left(\tilde C_1(n,\alpha,\kappa)\frac{E_n}{n^2}+\tilde C_2(n,\alpha,\kappa)\frac{E_{n+1}}{(n+1)^2}\right).
\end{equation}

Observe that this inequality is identical to the one obtained in \cite[Proof~of~Lemma~1]{aujol2023fista} under the assumption that $F$ has a unique minimizer. The value of $\tilde C_1(n,\alpha,\kappa)$ does not change while $\tilde C_2(n,\alpha,\kappa)$ is slightly larger (in the case of uniqueness of the minimizer, $\tilde C_2(n,\alpha,\kappa)$ is equal to $\hat C_2(n,\alpha,\kappa)$).  As a consequence, the bounds computed for $\tilde C_1(n,\alpha,\kappa)$ in \cite{aujol2023fista} are still valid and in particular, there exist some real constants $\tilde c_1$ and $\tilde c_2$ such that for any $\alpha\geqslant 3+\frac{3}{\sqrt{2}}$ and any $n\geqslant\frac{4\alpha}{3\sqrt{\kappa}}$,
\begin{equation}
\tilde C_1(n,\alpha,\kappa) \leqslant \frac{5}{4} \sqrt{\frac{2}{\kappa}} P(\alpha)\left(1+ \tilde{c}_1 \sqrt \kappa{+ \tilde{c}_2 \kappa}\right),
\end{equation} 
where $P:\alpha\mapsto \frac{2}{9}(\alpha-3)(\frac{8}{5}\alpha-3)-1$. Moreover, note that for any $n\geqslant \frac{4\alpha}{3\sqrt{\kappa}}$ and $\alpha\geqslant3$,
\begin{equation*}
\begin{aligned}
\widetilde C_2(n,\alpha,\kappa) &=\left(\frac{2\alpha}{3}-1\right)\left(8 +\frac{\sqrt{2}}{\sqrt{\kappa}}\right)\left(1+\frac{4\alpha^2}{9\kappa (n+1)^2}\right)\\
&\leqslant  \frac{5}{4}\sqrt{\frac{2}{\kappa}}\left(\frac{2\alpha}{3}-1\right) \left(1+4 \sqrt{2\kappa}\right).
\end{aligned}
\end{equation*}
Hence, for any $\alpha\geqslant 3+\frac{3}{\sqrt{2}}$:
\begin{equation}\label{eq:FISTA_ineq_12}
\forall n\geqslant \frac{4\alpha}{3\sqrt{\kappa}},~E_{n+1}-\left(1-\frac{\frac{2\alpha}{3}-2}{n}\right) E_n \leqslant \frac{\mathbf{C_1}(\alpha,\kappa)E_n}{n^2}+\frac{\mathbf{C_2}(\alpha,\kappa)E_{n+1}}{(n+1)^2},
\end{equation}
where:
\begin{itemize}
\item $\mathbf{C_1}(\alpha,\kappa) = \frac{5}{4} \sqrt{\frac{2}{\kappa}}\left[\frac{2}{9}(\alpha-3)\left(\frac{8}{5}\alpha -3\right)-1\right](1+\sqrt{\kappa})^2\left(1+ \tilde{c}_1 \sqrt \kappa{+ \tilde{c}_2 \kappa}\right)$,
\item $\mathbf{C_2}(\alpha,\kappa) = \frac{5}{4}\sqrt{\frac{2}{\kappa}} \left(\frac{2\alpha}{3}-1\right)(1+\sqrt{\kappa})^2(1+4 \sqrt{2\kappa}).$
\end{itemize}

From there, we refer the reader to \cite{aujol2023fista} since the last steps of this proof are detailed in the proof of \cite[Theorem 6]{aujol2023fista}.  We first integrate inequality \eqref{eq:FISTA_ineq_12} with the following lemma which is a slightly modified version of \cite[Lemma~2]{aujol2023fista}.
\begin{lemme}
\label{lem:nrj_fista}
Let $\alpha\geqslant3$ and $n_0\geqslant \frac{4\alpha}{3\sqrt{\kappa}}$. If the energy $E_n$ satisfies \eqref{eq:FISTA_ineq_12} then:
\begin{equation}
\forall n\geqslant n_0,~E_n\leqslant E_{n_0}\left(\frac{n}{n_0}\right)^{-\left(\frac{2\alpha}{3}-2\right)}e^{\phi(n_0)},
\end{equation}
where $\phi(n_0)=\frac{5}{6n_0}\sqrt{\frac{2}{\kappa}}(\alpha-3)\left(\frac{16}{15}\alpha-1\right)\left(1+c\kappa^{\frac{1}{4}}\right)$ and $c>0$ is independent to $\alpha$.
\end{lemme}
The proof of this lemma is identical to the proof of \cite[Lemma~2]{aujol2023fista} despite $\mathbf{C_2}(\alpha,\kappa)$ being larger than $C_2(\alpha,\kappa)$ in the other version. This difference is absorbed in the constant $c>0$.

Since $F(x_n)-F^*\leqslant \frac{L}{2n^2}E_n$, we get that for any $n\geqslant\frac{4\alpha}{3\sqrt{\kappa}}$,
\begin{equation*}
F(x_n)-F^*\leqslant \frac{L}{2}\left(n_0^{\frac{2\alpha}{3}-2}e^{\phi(n_0)}\right)E_{n_0}n^{-\frac{2\alpha}{3}}.
\end{equation*}
It is then essential to choose a relevant value for $n_0$ to get a control as tight as possible on $F(x_n)-F^*$. This discussion is already detailed in \cite{aujol2023fista} leading to the choice $$n_0=\frac{5}{4}\sqrt{\frac{2}{\kappa}}\left(\frac{16}{15}\alpha-1\right)\left(1+c\kappa^\frac{1}{4}\right),$$ which ensures that if $\kappa$ is sufficiently small, then
\begin{equation}
\forall n\geqslant \frac{3\alpha}{\sqrt{\kappa}},~F(x_n)-F^*\leqslant\frac{9}{4}e^{-2}M_0\left(\frac{8e}{3\sqrt{\kappa}}\alpha\right)^\frac{2\alpha}{3}n^{-\frac{2\alpha}{3}},
\end{equation}
where $M_0=F(x_0)-F^*$.\\
We now prove the second claim of Theorem \ref{thm:FISTA}. According to \eqref{eq:ineq_delta_n},  for any $n>\lambda$,
\begin{equation}
\delta_n\leqslant\frac{2}{(n-\lambda)^2}b_n+\frac{8\alpha^2}{9(n-\lambda)^2}h_n,
\end{equation}
where $b_n=\|\lambda (x_{n-1}-x_{n-1}^*)+n(x_n-x_{n-1})\|^2$. Considering the definition of the Lyapunov energy $E_n$, we have for any $n>\lambda$, $b_n\leqslant E_n$, hence:
\begin{equation}
\delta_n\leqslant \frac{2}{(n-\lambda)^2}E_n+\frac{8\alpha^2}{9(n-\lambda)^2}h_n,
\end{equation}
Since $F$ is assumed to satisfy a global quadratic growth condition $\mathcal{G}^2_\mu$ which implies that $h_n\leqslant \frac{E_n}{\kappa n^2}$ for any $n\in \N$, we get:
\begin{equation}
\forall n>\lambda,~\delta_n\leqslant \frac{2}{(n-\lambda)^2}\left(1+\frac{4\alpha^2}{9\kappa n^2}\right)E_n.
\end{equation}
Hence, for any $n\geqslant \frac{4\alpha}{3\sqrt{\kappa}}$, $\delta_n\leqslant \frac{5}{2(n-\lambda)^2}E_n$. By applying Lemma \ref{lem:nrj_fista}, we get that there exists some real constant $K>0$ such that $\delta_n\leqslant\frac{K}{n^\frac{2\alpha}{3}}$, which ensures that
\begin{equation}
\|x_n-x_{n-1}\|=\mathcal{O}\left(n^{-\frac{\alpha}{3}}\right).
\end{equation}

Finally, the strong convergence of FISTA iterates in the case when $F$ satisfies some global quadratic growth condition, follows from the summability of $\|x_n-x_{n-1}\|$ since $\alpha\geqslant 3+\frac{3}{\sqrt{2}}>3$.

\appendix
\section{Appendix}\label{app:1}

\subsection{Handling non-uniqueness of the minimizers in the continuous setting}
\label{sec:discussion_NU}
\off{
A classical approach to study the convergence of the trajectories of \eqref{eq:Nesterov_ODE} is to use a Lyapunov analysis. If $F$ has a unique minimizer i.e $X^*=\{x^*\}$, then a convenient Lyapunov energy is the following
\begin{equation}
\mathcal{E}(t)=t^2\left(F(x(t))-F^*\right)+\frac{1}{2}\|\lambda(x(t)-x^*)+t\dot{x}(t)\|^2+\frac{\xi}{2}\|x(t)-x^*\|^2,
\label{eq:Lyap_cont_uni}
\end{equation}
where $\lambda>0$ and $\xi\leqslant0$.

Our approach to extend classical analyses without the uniqueness assumption is to adapt the Lyapunov energies to our relaxed setting. Let $F$ have a non-empty set of minimizers $X^*$ which is not reduced to a point. We introduce the following Lyapunov function in this setting:}
In this section we assume that $F$ is a convex differentiable function having a $L$-Lipschitz gradient and a non-empty set of minimizers $X^*$. We introduce the following Lyapunov energy:
\begin{equation}
\mathcal{E}(t)=t^2\left(F(x(t))-F^*\right)+\frac{1}{2}\|\lambda(x(t)-x^*(t))+t\dot{x}(t)\|^2+\frac{\xi}{2}\|x(t)-x^*(t)\|^2,\label{eq:Lyap_disc}
\end{equation}
where for all $t\geqslant t_0$, $x^*(t)$ denotes the projection of $x(t)$ onto $X^*$, i.e
$$x^*(t)=\arg \inf\limits_{x^*\in X^*}\|x(t)-x^*\|^2.$$

Assume additionally that $X^*$ is second-order regular in the sense of Definition~\ref{def:second_order_regular} to that the projection $t\mapsto x^*(t)$ onto $X^*$ is right differentiable, as well as $\mathcal E$, and the right-hand derivative of $x^*$ is equal to $P^\prime_{X^*}(x(t),\dot{x}(t))$. For the sake of simplicity, let $\dot{x^*}$ and $\dot{\mathcal{E}}$ denote the corresponding right-hand derivatives. We can then write that:
\begin{equation}
\dot{\mathcal{E}}(t)=D(t)-(\lambda^2+\xi)\langle x(t)-x^*(t),\dot{x^*}(t)\rangle-\lambda t \langle \dot{x}(t),\dot{x^*}(t)\rangle,
\end{equation}
where 
$$\begin{aligned}D(t)=&2t\left(F(x(t))-F^*\right)+t^2\langle \nabla F(x(t)),\dot{x}(t)\rangle +\langle\lambda(x(t)-x^*(t))+t\dot{x}(t),(\lambda+1)\dot{x}(t)+t\ddot{x}(t)\rangle\\&+\xi\langle x(t)-x^*(t),\dot{x}(t)\rangle.\end{aligned}$$
Observe that $D$ is exactly equal to $\dot{\mathcal{E}}$ if $F$ has a unique minimizer $x^*$. The objective is then to control the additional terms $\langle x(t)-x^*(t),\dot{x^*}(t)\rangle$ and $\langle \dot{x}(t),\dot{x^*}(t)\rangle$. We introduce Figure \ref{fig:minimizers} to give an intuition of the behavior of these terms.

\begin{figure}[h]
\begin{center}
\begin{tikzpicture}
	\draw [fill=gray!20,domain=-2.5:2.5] plot (-\x*\x/3,\x) ;
	\draw (-1.2,-1.5) node[left]{$X^*$} ;
	\draw (2/3,5/3) node[right] {$x(t)$} node{$\bullet$} ;
	\draw (-1/3,1) node[below left] {$x^*(t)$} node{$\bullet$};
	\draw [<-,red,very thick] (1/6,2) -- (2/3,5/3) ;
	\draw [red] (1/6,2) node[above]{\small$\dot{x}(t)$\normalsize};
	\draw [<-,red,very thick] (-2/3,1.5) -- (-1/3,1) ;
	\draw [red] (-2/3,1.5) node[below left]{\small$\dot{x^*}(t)$\normalsize};
	\draw [->,blue,very thick] (-1/3,1) -- (2/3,5/3);
	\draw [blue] (1,0.9) node[]{\small$x(t)-x^*(t)$\normalsize};
\end{tikzpicture}
\begin{tikzpicture}
	\draw [fill=gray!20,domain=-2.5:2.5] plot (\x,-{abs(\x)}) ;
	\draw (0.6,-1.9) node[right]{$X^*$} ;
	\draw (-0.3,0.8) node[above right] {$x(t)$} node{$\bullet$} ;
	\draw (0,0) node[right] {$x^*(t)$} node{$\bullet$};
	\draw [<-,red,very thick] (-0.6,0.4) -- (-0.3,0.8) ;
	\draw [red] (-0.6,0.4) node[above left]{\small$\dot{x}(t)$\normalsize};
	\draw [red] (0,0) node[below left]{\small$\dot{x^*}(t)=0$\normalsize};
	\draw [->,blue,very thick] (0,0) -- (-0.3,0.8);
	\draw [blue] (0,0.5) node[right]{\small$x(t)-x^*(t)$\normalsize};
\end{tikzpicture}
\end{center}
\caption{Behavior of $\dot{x^*}$ for a set of minimizers having a $C^2$ bound (on the left) and a polyhedral set of minimizers (on the right).}
\label{fig:minimizers}
\end{figure}
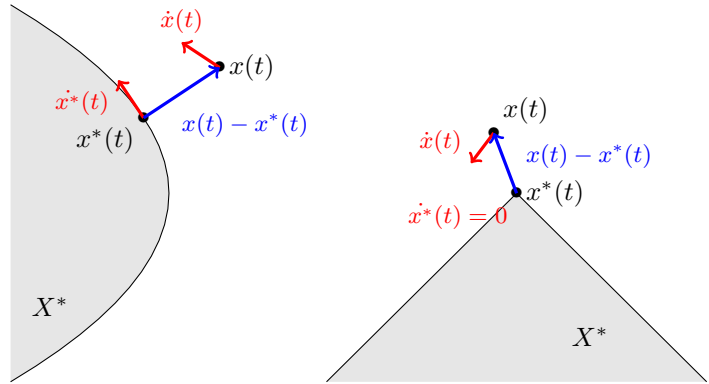

We can first prove that $\langle \dot{x}(t),\dot{x^*}(t)\rangle$ is positive by using the expression $\dot{x^*}(t)=\lim\limits_{h\rightarrow0}\frac{x^*(t+h)-x^*(t)}{h}$ and the property of the projection onto a convex set. Indeed, as $X^*$ is a closed convex set, for any $x\in\mathcal{H}$ and $u\in X^*$:
\begin{equation*}
\langle x-P_{X^*}(x),u-P_{X^*}(x)\rangle\leqslant0.
\end{equation*}
Thus, for any $h>0$ we have:\small
\begin{align*}
\langle x(t+h)-x(t),x^*(t+h)-x^*(t)\rangle&=\langle x(t+h)-x^*(t+h),x^*(t+h)-x^*(t)\rangle\\&~~~+\|x^*(t+h)-x^*(t)\|^2\\&~~~+\langle x(t)-x^*(t),x^*(t)-x^*(t+h)\rangle\\&\geqslant0.
\end{align*}\normalsize
By considering $h$ tending towards $0$ we can deduce that $\langle \dot{x}(t),\dot{x^*}(t)\rangle\geqslant0$.

In \cite[Theorem~7.2]{bonnans1998sensitivity} the authors give an expression of the directional derivative $P'_\mathcal{S}(x,d)$ for a closed convex set $\mathcal{S}\subset\mathcal{X}$ being second order regular at $P_\mathcal{S}(x)$ for some $x\in\mathcal{X}$. This directional derivative satisfies:
\begin{equation*}
\langle x-P_\mathcal{S}(x),P'_\mathcal{S}(x,d)\rangle=0.
\end{equation*}
Considering the assumptions made on $X^*$ we can deduce that $\langle x(t)-x^*(t),\dot{x^*}(t)\rangle=0$ for all $t\geqslant t_0$. 

These results ensure that for any choices of parameters $\lambda>0$ and $\xi\in\R$, we have that $\dot{\mathcal{E}}^*(t)\leqslant D(t)$. From this point, it is sufficient to apply the following lemma to extend the desired convergence results to the non-unique case. A proof is given in Section \ref{sec:proof_right_diff}.

\begin{lemme}
\label{lem:right-diff}
Let $\phi:\R\rightarrow\R$ be a continuous function which is right-differentiable. Assume that 
\begin{equation}
\forall t\geqslant t_0,~\phi_+(t)\leqslant\psi(t),
\label{eq:ineq_phi+}
\end{equation}
where $\phi_+(t)=\displaystyle\lim\limits_{h\rightarrow0,~h>0}\frac{\phi(t+h)-\phi(t)}{h}$ denotes the right derivative of $\phi$ at $t$. Then,
\begin{equation}
\forall t\geqslant t_0,~\phi(t)\leqslant\phi(t_0)+\int_{t_0}^t\psi(u)du.
\label{eq:maj_semidif}
\end{equation}
\end{lemme}

\subsection{Proof of Theorem \ref{thm:AVD_flat} under Hölderian growth condition}
\label{sec:proof_flat}
We focus our analysis on the following Lyapunov energy introduced in \cite{sebbouh2020convergence}:
\begin{equation}
\mathcal{J}(t)=t^p\left(t^2(F(x(t))-F^*)+\frac{1}{2}\|\lambda(x(t)-x^*(t))+t\dot{x}(t)\|^2+\frac{\xi}{2}\|x(t)-x^*(t)\|^2\right),
\end{equation}
where $p=1+\frac{4}{\gamma-2}$ and $\lambda>0$. We use the following notations:
\begin{gather*}
a(t)=t\left(F(x(t))-F^*\right),~b(t)=\frac{1}{2t}\|\lambda (x(t)-x^*(t))+t\dot{x}(t)\|^2\\
c(t)=\frac{1}{2t}\|x(t)-x^*(t)\|^2.
\end{gather*}
The Lyapunov function can be rewritten as follows:
\begin{equation*}
\mathcal{J}(t)=t^{p+1}\left(a(t)+b(t)+\xi c(t)\right).
\end{equation*}
Following the discussion on the derivability of $\mathcal{E}^*$ defined in \eqref{eq:Lyap_disc} in Section \ref{sec:discussion_NU}, we can say that under the assumption made on $X^*$, $\mathcal{E}^*$ is right differentiable. Noticing that $\mathcal{J}(t)=t^p\mathcal{E}^*(t)$, this is also true for $\mathcal{J}$. For the sake of simplicity, the right derivative of $\mathcal{J}$ is denoted $\mathcal{J}^\prime$. By adapting \cite[Lemma~4.4]{sebbouh2020convergence} to our case, we get that if $\xi=\lambda(\lambda+1-\alpha)$, then
\begin{equation*}
\mathcal{J}^\prime(t)\leqslant t^p\left((2+p-\lambda)a(t)+(2(\lambda+1-\alpha)+p)b(t)+\lambda(\lambda+1-\alpha)(p-2\lambda)c(t)\right).
\end{equation*}
Let $\lambda=\alpha-1-\frac{p}{2}$. Under the condition $\alpha>\frac{9}{2}+\frac{6}{\gamma-2}$, we have that 
\begin{equation*}
\left\{\begin{gathered}
2+p-\lambda<0,\\
2(\lambda+1-\alpha)+p=0,\\
\lambda(\lambda+1-\alpha)(p-2\lambda)>0.\end{gathered}\right.
\end{equation*}
As a consequence, we can write that:
\begin{equation*}
\mathcal{J}^\prime (t)\leqslant t^p\left(Aa(t)+Cc(t)\right),
\end{equation*}
where $A=3-\alpha+\frac{3p}{2}<0$ and $C=p\left(\alpha-1-\frac{p}{2}\right)(\alpha-1)>0$. We can apply \cite[Lemma~4.5]{sebbouh2020convergence} which we recall below.
\begin{lemme}
\label{lem:flat_cond}
If $F$ satisfies the inequality \eqref{eq:Lojasiewicz_dis_flat} for some $\gamma>2$ and $K>0$, i.e. $F$ satisfies $\mathcal{G}^\gamma_{loc}$, then there exists $t_1\geqslant t_0$ such that for all $t\geqslant t_1$,
\begin{equation*}
t^{p_2+1}c(t)\leqslant \frac{K^{-\frac{2}{\gamma}}}{2}\left(t^{p_2+1}a(t)\right)^\frac{2}{\gamma},
\end{equation*}
where $p_2=\frac{4}{\gamma-2}$.\\
It follows that for any $m\in\R$, there exists $M\in\R$ such that for any $t\geqslant t_1$,
\begin{equation*}
t^{p_2+1}(m c(t)-a(t))\leqslant M.
\end{equation*}
\end{lemme}
As $p=p_2+1$, this lemma ensures that there exists $M_1\in\R$ such that for all $t\geqslant t_0$, $\mathcal{J}^\prime(t)\leqslant M_1$. Then, Lemma \ref{lem:right-diff} gives us that that there exists $M_2\in\R$ such that $\mathcal{J}(t)\leqslant M_1t+M_2$ and consequently
\begin{equation*}
t^{p+1}(a(t)+\xi c(t))\leqslant M_1t+M_2.
\end{equation*}
Therefore, for $t$ sufficiently large,
\begin{equation*}
t^pa(t)\leqslant 2M_1+|\xi| t^p c(t).
\end{equation*}
The first claim of Lemma \ref{lem:flat_cond} gives us that there exists $M_3>0$ such that:
\begin{equation}
t^pa(t)\leqslant 2M_1+M_3\left(t^pa(t)\right)^\frac{2}{\gamma}.
\end{equation}
Lemma \ref{lem:control_v} guarantees that there exists $M_4>0$ such that for $t$ sufficiently large, 
\begin{equation*}
t^pa(t)\leqslant M_4,
\end{equation*}
and thus,
\begin{equation}
F(x(t))-F^*\leqslant \frac{M_4}{t^{p+1}}.
\end{equation}
As $p+1=\frac{2\gamma}{\gamma-2}$, the first claim is proved.\\
We prove the second claim by coming back to the inequality $\mathcal{J}(t)\leqslant M_1t+M_2$. By applying the inequality $\|u\|^2\leqslant2\|u+v\|^2+2\|v\|^2$ to $u=t\dot{x}(t)$ and $v=\lambda(x(t)-x^*(t))$, we get that
\begin{equation*}
b(t)\geqslant\frac{t}{2}\|\dot{x}(t)\|^2-\lambda^2c(t).\end{equation*}
Consequently, for sufficiently large $t$ we have that:
\begin{equation*}
t^p\left(a(t)-\left(|\xi|+\lambda^2\right)c(t)+\frac{t}{2}\|\dot{x}(t)\|^2\right)\leqslant 2M_1.
\end{equation*}
Lemma \ref{lem:flat_cond} gives us that there exists $M_5>0$ such that:
\begin{equation*}
t^p\left(a(t)-\left(|\xi|+\lambda^2\right)c(t)\right)\geqslant t^p a(t)-M_5\left(t^pa(t)\right)^\frac{2}{\gamma}.
\end{equation*}
\begin{lemme}
\label{lem:min_g}
Let $g:x\mapsto x-Kx^\delta$ for some $K>0$ and $\delta\in(0,1)$. Then for all $x\geqslant 0$,
$$g(x)\geqslant K(\delta-1)(\delta K)^{\frac{\delta}{1-\delta}}.$$
\end{lemme}
Lemma \ref{lem:min_g} ensures that there exists $M_6\in\R$ such that $t^p\left(a(t)-\left(|\xi|+\lambda^2\right)c(t)\right)\geqslant M_6$. Hence, for $t$ sufficiently large,
\begin{equation*}
\frac{t^{p+1}}{2}\|\dot{x}(t)\|^2\leqslant 2M_1+M_6,
\end{equation*}
and thus:
\begin{equation}
\|\dot{x}(t)\|\leqslant \frac{M_7}{t^{\frac{p+1}{2}}},
\end{equation}
where $M_7=4M_1+2M_6\geqslant0$ and $\frac{p+1}{2}=\frac{\gamma}{\gamma-2}$. Thus the trajectory $t\mapsto (x(t),\dot x(t))$ has a finite length and $t\mapsto x(t)$ strongly converges to a minimizer of $F$.

\subsection{Proof of Theorem \ref{thm:continu1} and Proposition \ref{thm:continu2} under a quadratic growth condition}
\label{sec:proof_traj_cont1}
The proof of Theorem \ref{thm:continu1} is very similar to the one of \cite[Theorem~5]{aujol2023fista} and is not reproduced entirely here, but we recall the essential steps of this proof. We first introduce the following Lyapunov energy:
\begin{equation}
\mathcal{E}(t)=t^2(F(x(t))-F^*)+\frac{1}{2}\|\lambda(x(t)-x^*(t))+t\dot{x}(t)\|^2,~\lambda=\frac{2\alpha}{3}
\end{equation}
where $x^*(t)$ denotes the projection of the trajectory $x(t)$ solution of $\eqref{AVD}$ onto the set of minimizers. According to the discussion in Section \ref{sec:discussion_NU}, the energy $\mathcal E$ is (right-)differentiable, allowing to carry out the proof of \cite[Theorem~5]{aujol2023fista} without any particular difficulty. The only additional challenge is to deal with the terms involving $\dot x^*(t)$, which is described in Section \ref{sec:discussion_NU}).
\paragraph{\it First case: $F$ satisfies a global quadratic growth condition (Proposition \ref{thm:continu2})}
Following the proof of \cite[Theorem~5]{aujol2023fista}, we can show that the right derivative of $\mathcal{E}$ denoted $\mathcal{E}^\prime$ satisfies:
\begin{equation*}
    \forall t\geqslant t_0,\quad\mathcal{E}^\prime(t)+\frac{\lambda-2}{t}\mathcal{E}(t)\leqslant\phi(t)\mathcal{E}(t),
\end{equation*}
where:
\begin{equation*}
    \phi:t\mapsto \frac{2\alpha\left(\alpha-3\right)}{9\mu t^2}\left(\sqrt{\mu}+\frac{2\alpha}{3t}(1+\sqrt{2})+\frac{4\alpha^2}{9\sqrt{\mu}t^2}\right).
\end{equation*}
This inequality combined with Lemma \ref{lem:right-diff} ensures that $t\mapsto \mathcal{E}(t)t^{\lambda-2}e^{\Phi(t)}$, where $\Phi:t\mapsto\int_t^{+\infty}\phi(s)ds$, is decreasing on $[t_0,+\infty)$. As a consequence, for any $t_1\geqslant t_0$:
\begin{equation*}
    \forall t\geqslant t_1,\quad \mathcal{E}(t)\leqslant\mathcal{E}(t_1)\left(\frac{t_1}{t}\right)^{\lambda-2}e^{\Phi(t_1)-\Phi(t)}.
\end{equation*}
The next steps of the demonstration rely on showing that $\Phi$ is positive, choosing a relevant value for $t_1$ and bounding each term of the inequality.
\paragraph{\it Second case: $F$ satisfies a local quadratic growth condition (Theorem \ref{thm:continu1})}
Similar to Lemma \ref{lem_cv_locale}, we can use the coercivity of $F$ and the convergence of $t\mapsto F(x(t))-F^*$ to $0$ (since it is well known that for any $\alpha>3$, $F(x(t))-F^*=\mathcal{O}\left(t^{-2}\right)$, see \cite{su2014differential}) to prove the existence of $t_\varepsilon\geqslant t_0$ such that:
\begin{equation*}
    \forall t\geqslant t_\varepsilon,\quad \frac{\mu}{2}d(x(t),X^*)\leqslant F(x(t))-F^*).
\end{equation*}
By following the proof in the global case and replacing $t_0$ by $t_\varepsilon$, we can easily find the desired asymptotic result:
\begin{equation*}
    F(x(t))-F^*=\mathcal{O}\left(t^{-\frac{2\alpha}{3}}\right).
\end{equation*}
\paragraph{\it Showing that the trajectory has a finite length}
We consider that $F$ satisfies $\mathcal{G}^2_{\mu,loc}$. It is shown that there exist some $t_1\geqslant t_\varepsilon$ and $K>0$ such that
\begin{equation}
\forall t\geqslant t_1,~\mathcal{E}(t)\leqslant Kt^{-\frac{2\alpha}{3}+2}.\label{eq:continu_AVD_nrj}
\end{equation}
Moreover, by applying inequality $\|u\|^2\leqslant 2\|u+v\|^2+2\|v\|^2$, we obtain that:
\begin{equation}
\|\dot{x}(t)\|^2\leqslant \frac{2}{t^2}\|\lambda(x(t)-x^*(t))+t\dot{x}(t)\|^2+\frac{2\lambda^2}{t^2}\|x(t)-x^*(t)\|^2.
\end{equation}
Hence, the assumption $\mathcal{G}^2_{\mu,loc}$ guarantees that
\begin{equation}
\forall t\geqslant t_1,\quad\|\dot{x}(t)\|^2\leqslant \frac{4}{t^2}\left(1+\frac{\lambda^2}{\mu t^2}\right)\mathcal{E}(t).
\end{equation}
Inequality \eqref{eq:continu_AVD_nrj} gets us to the conclusion:
\begin{equation}
\|\dot{x}(t)\|=\mathcal{O}\left(t^{-\frac{\alpha}{3}}\right).
\end{equation}
Since $\alpha>3$, we obtain that $\int_{t_1}^{+\infty}\|\dot{x}(t)\|dt<+\infty$ which implies that the trajectory $x(\cdot)$ has a finite length. Combined with the convergence rate on function values, this guarantees that $x(\cdot)$ converges to some minimizer of $F$.
\qed

\section{Proofs of technical Lemmas \ref{lem:tech1}, \ref{lem:tech2}, \ref{lem:ineq_flat1}, \ref{lem:lemma_FISTA_AB} and \ref{lem:right-diff}}\label{app:2}

\subsection{Proof of Lemma \ref{lem:tech1}}
\label{sec:proof_lemma_tech1}
Let $n\in\N^*$. By rewriting \begin{equation*}x_n-x_n^*=\frac{1}{2}\left((x_n-x_{n-1})+(x_{n-1}-x_{n-1}^*)+(x_{n-1}^*-x_n^*)+(x_n-x_n^*)\right),\end{equation*}we get that:
\begin{equation*}
\langle x_n-x_n^*,x_n-x_{n-1}\rangle=\frac{1}{2}\delta_n+\frac{1}{2}\langle(x_{n-1}-x_{n-1}^*)+(x_{n-1}^*-x_n^*)+(x_n-x_n^*),x_n-x_{n-1}\rangle.
\end{equation*}
Noticing that $x_n-x_{n-1}=(x_n-x_n^*)+(x_n^*-x_{n-1}^*)+(x_{n-1}^*-x_{n-1})$ leads to:
\begin{equation*}
\begin{aligned}
2\langle x_n-x_n^*,x_n-x_{n-1}\rangle&=\delta_n+\langle x_{n-1}-x_{n-1}^*,x_n-x_n^*\rangle +\langle x_{n-1}-x_{n-1}^*,x_n^*-x_{n-1}^*\rangle\\
&-h_{n-1}-\langle x_n^*-x_{n-1}^*,x_n-x_n^*\rangle+\langle x_{n-1}-x_{n-1}^*,x_n^*-x_{n-1}^*\rangle\\
&-\gamma_n^*+\langle x_{n-1}-x_{n-1}^*,x_n^*-x_{n-1}^*\rangle-\langle x_{n-1}-x_{n-1}^*,x_n-x_n^*\rangle+h_n\\
&=h_n-h_{n-1}+\delta_n-\gamma_n^*+2\langle x_{n-1}-x_{n-1}^*,x_n^*-x_{n-1}^*\rangle.
\end{aligned}
\end{equation*}
The second claim is proved using the same approach. We rewrite \begin{equation*}x_{n-1}-x_{n-1}^*=\frac{1}{2}\left((x_{n-1}-x_n)+(x_n-x_n^*)+(x_n^*-x_{n-1}^*)+(x_{n-1}^*-x_{n-1})\right),\end{equation*}and consequently:
\begin{equation*}
2\langle x_{n-1}-x_{n-1}^*,x_n-x_{n-1}\rangle=-\delta_n+\langle (x_n-x_n^*)+(x_n^*-x_{n-1}^*)+(x_{n-1}^*-x_{n-1}),x_n-x_{n-1}\rangle.
\end{equation*}
By applying the same rewriting of $x_n-x_{n-1}$, simple calculations give that:
\begin{equation*}
\langle x_{n-1}-x_{n-1}^*,x_n-x_{n-1}\rangle=\frac{1}{2}(h_n-h_{n-1}-\delta_n+\gamma_n^*)+\langle x_{n}-x_{n}^*,x_n^*-x_{n-1}^*\rangle.
\end{equation*}\qed

\subsection{Proof of Lemma \ref{lem:tech2}}
\label{sec:proof_lemma_tech2}
The first claim is straightforward as Lemma 3.1 of \cite{chambolle2015convergence} ensures that:
\begin{equation*}
F(x_{n+1})-F(x_n)\leqslant\frac{L}{2}\left(\|y_n-x_n\|^2-\|x_{n+1}-x_n\|^2\right).
\end{equation*}
By writing $y_n=x_n+\alpha_n(x_n-x_{n-1})$ and $\frac{2}{L}(F(x_{n+1})-F(x_n))=w_{n+1}-w_n$, we can conclude.

By applying Lemma 3.1 of \cite{chambolle2015convergence} to an other couple of points, we get that:
\begin{equation*}
F(x_{n+1})-F^*\leqslant \frac{L}{2}\left(\|y_n-x_n^*\|^2-\|x_{n+1}-x_n^*\|^2\right).
\end{equation*}
It follows that:
\begin{equation*}
\begin{aligned}
w_{n+1}&\leqslant\|x_n+\alpha_n(x_n-x_{n-1})-x_n^*\|^2-\|(x_{n+1}-x_{n+1}^*)+(x_{n+1}^*-x_n^*)\|^2\\
&\leqslant h_n+\alpha_n^2\delta_n-h_{n+1}-\gamma_{n+1}^*+2\alpha_n\langle x_n-x_n^*,x_n-x_{n-1}\rangle\\&-2\langle x_{n+1}-x_{n+1}^*,x_{n+1}^*-x_n^*\rangle.
\end{aligned}
\end{equation*}
Recall that the first claim of Lemma \ref{lem:tech1} ensures that:
\begin{equation*}
\langle x_n-x_n^*,x_n-x_{n-1}\rangle=\frac{1}{2}(h_n-h_{n-1}+\delta_n-\gamma_n^*)+\langle x_{n-1}-x_{n-1}^*,x_n^*-x_{n-1}^*\rangle,
\end{equation*}
we can deduce that:
\begin{equation*}
\begin{aligned}
w_{n+1}&\leqslant (1+\alpha_n)h_n+(\alpha_n^2+\alpha_n)\delta_n-\alpha_n h_{n-1}-h_{n+1}-\gamma_{n+1}^*-\alpha_n\gamma_n^*\\
&+2\alpha_n\langle x_{n-1}-x_{n-1}^*,x_n^*-x_{n-1}^*\rangle-2\langle x_{n+1}-x_{n+1}^*,x_{n+1}^*-x_n^*\rangle.
\end{aligned}
\end{equation*}\qed

\subsection{Proof of Lemma \ref{lem:ineq_flat1}}
\label{sec:proof_tech_flat1}

Recall the definition of the discrete Lyapunov energy $\mathcal{E}$:
\begin{equation}
\mathcal{E}_n=n^2w_n+b_n+\xi h_n+\lambda n \alpha_n^2\delta_n.
\end{equation}
Observe that for any $n\in\N$,
\begin{equation}
b_n=\lambda^2 h_n+n^2\alpha_n^2\delta_n+2\lambda n\alpha_n\langle x_n-x_n^*,x_n-x_{n-1}\rangle,
\end{equation}
and by applying the first claim of Lemma \ref{lem:tech1} we get that:
\begin{equation}
\begin{aligned}
b_n=&\lambda^2 h_n+\lambda n\alpha_n(h_n-h_{n-1})+n\alpha_n(n\alpha_n+\lambda)\delta_n\\&-\lambda n\alpha_n \gamma_n^*+2\lambda n\alpha_n\langle x_{n-1}-x_{n-1}^*,x_n^*-x_{n-1}^*\rangle.
\end{aligned}
\label{eq:dev_bn}
\end{equation}
As a consequence,\small
\begin{equation}
\begin{aligned}
b_{n+1}-b_n = &\lambda(\lambda+(n+1)\alpha_{n+1})(h_{n+1}-h_n)+(n+1)\alpha_{n+1}\left((n+1)\alpha_{n+1}+\lambda\right)\delta_{n+1}\\
&-\lambda n\alpha_n(h_n-h_{n-1})-n\alpha_n(n\alpha_n+\lambda)\delta_n-\lambda (n+1)\alpha_{n+1}\gamma_{n+1}^*+\lambda n\alpha_n\gamma_n^*\\
&+2\lambda(n+1)\alpha_{n+1}\langle x_n-x_n^*,x_{n+1}^*-x_n^*\rangle-2\lambda n \alpha_n\langle x_{n-1}-x_{n-1}^*,x_n^*-x_{n-1}^*\rangle.
\end{aligned}\label{eq:b_n1}
\end{equation}\normalsize
On the other hand, we have that:
\begin{equation}
(n+1)^2w_{n+1}-n^2w_n=n^2(w_{n+1}-w_n)+(2n+1)w_{n+1},
\end{equation}
and by applying the first claim of Lemma \ref{lem:tech2}:
\begin{equation}
(n+1)^2w_{n+1}-n^2w_n\leqslant n^2(\alpha_n^2\delta_n-\delta_{n+1})+(2n+1)w_{n+1},\label{eq:w_n1}
\end{equation}
By combining \eqref{eq:b_n1} and \eqref{eq:w_n1}, we get that:
\begin{equation}
\begin{aligned}
\mathcal{E}_{n+1}-\mathcal{E}_n\leqslant &(2n+1)w_{n+1}-\lambda n(\alpha_n+\alpha_n^2)\delta_n\\&+((n+1)\alpha_{n+1}((n+1)\alpha_{n+1}+\lambda)+\lambda(n+1)\alpha_{n+1}^2-n^2)\delta_{n+1}\\
&+\lambda n\alpha_n (h_{n-1}-h_n)-\lambda\left(\lambda+(n+1)\alpha_{n+1}+\frac{\xi}{\lambda}\right)(h_n-h_{n+1})\\&-\lambda (n+1)\alpha_{n+1}\gamma_{n+1}^*+2\lambda(n+1)\alpha_{n+1}\langle x_n-x_n^*,x_{n+1}^*-x_n^*\rangle\\&+\lambda n\alpha_n\gamma_n^*-2\lambda n \alpha_n\langle x_{n-1}-x_{n-1}^*,x_n^*-x_{n-1}^*\rangle.
\end{aligned}\label{eq:in_E1}
\end{equation}
Observe that the second claim of Lemma \ref{lem:tech2} guarantees that:\footnotesize
\begin{equation}
\begin{gathered}
-\lambda n w_{n+1}+\lambda n(h_n-h_{n+1})+\lambda n\alpha_n(h_n-h_{n-1})+\lambda n (\alpha_n+\alpha_n^2)\delta_n-\lambda n \gamma_{n+1}^*-\lambda n \alpha_n\gamma_n^*\\+2\lambda n\alpha_n\langle x_{n-1}-x_{n-1}^*,x_n^*-x_{n-1}^*\rangle-2\lambda n \langle x_{n+1}-x_{n+1}^*,x_{n+1}^*-x_n^*\rangle\leqslant0.
\end{gathered}
\label{eq:ineq_lnw}
\end{equation}\normalsize
Adding inequality \eqref{eq:ineq_lnw} to \eqref{eq:in_E1} leads to
\begin{equation}
\begin{aligned}
\mathcal{E}_{n+1}-\mathcal{E}_n\leqslant &((2-\lambda)n+1)w_{n+1}+((n+1)\alpha_{n+1}((n+1)\alpha_{n+1}+\lambda)+\lambda(n+1)\alpha_{n+1}^2-n^2)\delta_{n+1}\\
&-\lambda\left(\lambda+(n+1)\alpha_{n+1}+\frac{\xi}{\lambda}-n\right)(h_n-h_{n+1})+\mathcal{X}_n^*,
\end{aligned}
\end{equation}
where \begin{equation*}\begin{gathered}\mathcal{X}_n^*=(-\lambda (n+1)\alpha_{n+1}-\lambda n)\gamma_{n+1}^*+2\lambda(n+1)\alpha_{n+1}\langle x_n-x_n^*,x_{n+1}^*-x_n^*\rangle\\-2\lambda n \langle x_{n+1}-x_{n+1}^*,x_{n+1}^*-x_n^*\rangle.\end{gathered}\end{equation*}
Observe that since $X^*$ is a closed convex set, $\langle x_n-x_n^*,x_{n+1}^*-x_n^*\rangle\leqslant0$ and $\langle x_{n+1}-x_{n+1}^*,x_{n+1}^*-x_n^*\rangle\geqslant0$. Hence, for any $n\in\N$, $\mathcal{X}_n^*\leqslant0$.
By choosing $\xi=\lambda(\lambda+1-\alpha)$, we then obtain that:
\begin{equation}
\mathcal{E}_{n+1}-\mathcal{E}_n\leqslant ((2-\lambda)n+1)w_{n+1}+A_1(n)\delta_{n+1}+A_2(n)(h_n-h_{n+1}),
\end{equation}
where:
\begin{itemize}
\item $A_1(n)=2(\lambda+1-\alpha)n+\frac{n^2(3\alpha^2-3\alpha\lambda-2\alpha+2\lambda+1)+n(2\alpha^3-2\alpha^2\lambda+2\alpha^2-2\alpha\lambda-2\alpha+4\lambda+2)+1+2\lambda+\alpha\lambda}{(n+1+\alpha)^2}$,
\item $A_2(n)=-2\lambda(\lambda+1-\alpha)-\frac{\alpha^2\lambda}{n+1+\alpha}$.
\end{itemize}
Note that by rewriting \eqref{eq:dev_bn} we get that:
\begin{equation}
h_{n-1}-h_n=-\frac{1}{\lambda n\alpha_n}b_n+\frac{\lambda}{n\alpha_n}h_n+\frac{n\alpha_n+\lambda}{\lambda}\delta_n-\gamma_n^*+2\langle x_{n-1}-x_{n-1}^*,x_n^*-x_{n-1}^*\rangle.
\end{equation}
This ensures that:\small
\begin{equation}
\begin{aligned}
\mathcal{E}_{n+1}-\mathcal{E}_n\leqslant &((2-\lambda)n+1)w_{n+1}-\frac{A_2(n)}{\lambda(n+1)\alpha_{n+1}}b_{n+1}+\frac{\lambda A_2(n)}{(n+1)\alpha_{n+1}}h_{n+1}\\&+\left(A_1(n)+\frac{(n+1)\alpha_{n+1}+\lambda}{\lambda}A_2(n)\right)\delta_{n+1}\\&-A_2(n)\left(\gamma_{n+1}^*-2\langle x_n-x_n^*,x_{n+1}^*-x_n^*\rangle\right),
\end{aligned}
\end{equation}\normalsize
which is the desired inequality.\\ \qed

\subsection{Proof of Lemma \ref{lem:ineq_flat2}}
\label{sec:proof_tech_flat2}
Let $\xi=\lambda(\lambda+1-\alpha)$. Let $\mathcal{J}_n=n^p\mathcal{E}_n$ with $p=1+\frac{4}{\gamma-2}$. Elementary computations show that:
\begin{equation}
\mathcal{J}_{n+1}-\mathcal{J}_n=n^p\left(\mathcal{E}_{n+1}-\mathcal{E}_n\right)+\left((n+1)^p-n^p\right)\mathcal{E}_{n+1}.
\end{equation}
Observe that for any $n\in\N$, $(n+1)^p-n^p\in\left[pn^{p-1},p(n+1)^{p-1}\right]$. Therefore, if we make the assumption that $\lambda\leqslant \alpha-1$, we obtain that $\xi\leqslant 0$ and:
\begin{equation}
\begin{aligned}
\left((n+1)^p-n^p\right)\mathcal{E}_{n+1}\leqslant &~p(n+1)^{p+1}w_{n+1}+p(n+1)^{p-1}b_{n+1}+p\xi n^{p-1}h_{n+1}\\&+p\lambda(n+1)^p\alpha_{n+1}^2\delta_{n+1}. 
\end{aligned}
\end{equation} 
By applying Lemma \ref{lem:ineq_flat1} and the above inequality we get that:
\begin{equation}
\begin{aligned}
\mathcal{J}_{n+1}-\mathcal{J}_n\leqslant &~\left(n^p\left((2-\lambda)n+1\right)+p(n+1)^{p+1}\right)w_{n+1}\\
&+\left(n^p B_1(n)+p(n+1)^{p-1}\right)b_{n+1}\\
&+\left(n^p B_2(n)+p\xi n^{p-1}\right)h_{n+1}\\
&+\left(n^p B_3(n)+p\lambda(n+1)^p\alpha_{n+1}^2\right)\delta_{n+1}\\
&-n^p B_4(n)\left(\gamma_{n+1}^*-2\langle x_n-x_n^*,x_{n+1}^*-x_n^*\rangle\right).
\end{aligned}
\end{equation}
The inequality $\|u\|^2\leqslant 2\|u+v\|^2+2\|v\|^2$ applied at $u=\alpha_n(x_n-x_{n-1})$ and $v=\lambda(x_n-x_n^*)$ ensures that:
\begin{equation}
\delta_n\leqslant \frac{2}{n^2\alpha_n^2}b_n+\frac{2\lambda^2}{n^2\alpha_n^2}h_n.
\end{equation}
Thus, 
\begin{equation}
\begin{aligned}
\mathcal{J}_{n+1}-\mathcal{J}_n\leqslant &~\left(n^p\left((2-\lambda)n+1\right)+p(n+1)^{p+1}\right)w_{n+1}\\
&+\left(n^p B_1(n)+p(n+1)^{p-1}+2\frac{n^p|B_3(n)|+p\lambda(n+1)^p\alpha_{n+1}^2}{(n+1)^2\alpha_{n+1}^2}\right)b_{n+1}\\
&+\left(n^p B_2(n)+p\xi n^{p-1}+2\lambda^2\frac{n^p|B_3(n)|+p\lambda(n+1)^p\alpha_{n+1}^2}{(n+1)^2\alpha_{n+1}^2}\right)h_{n+1}\\
&-n^p B_4(n)\left(\gamma_{n+1}^*-2\langle x_n-x_n^*,x_{n+1}^*-x_n^*\rangle\right).
\end{aligned}
\end{equation}
By replacing $\xi$ by its value and reorganizing each term, we get to the conclusion.\\ \qed

\subsection{Proof of Lemma \ref{lem:lemma_FISTA_AB}}
\label{sec:proof_lemma_FISTA_AB}
Let $(A,B)\in\R^2$. Elementary computations show that for any $n\in\N^*$,
\begin{equation*}
h_{n-1}-h_n-\delta_n=-2\langle x_n-x_{n-1},x_n-x_n^*\rangle+2\langle x_{n-1}-x_{n-1}^*,x_n^*-x_{n-1}^*\rangle-\gamma_n^*.
\end{equation*}
Consequently, for any $n\in\N^*$,
\begin{equation*}
\begin{aligned}
A\delta_n+B(h_{n-1}-h_n)&=(A+B)\delta_n+B(h_{n-1}-h_n-\delta_n)\\
&=(A+B)\delta_n-2B\langle x_n-x_{n-1},x_n-x_n^*\rangle\\&+2B\langle x_{n-1}-x_{n-1}^*,x_n^*-x_{n-1}^*\rangle-B\gamma_n^*\\
&\leqslant (A+B)\delta_n+2|B|\left|\langle x_n-x_{n-1},x_n-x_n^*\rangle\right|\\&+2B\langle x_{n-1}-x_{n-1}^*,x_n^*-x_{n-1}^*\rangle-B\gamma_n^*.
\end{aligned}
\end{equation*}
Moreover, note that for any $n\in\N^*$ and $\theta>0$:
\begin{equation}
2|\langle x_n-x_{n-1},x_n-x_n^*\rangle|\leqslant \frac{h_n}{\theta}+\theta\delta_n.  
\end{equation}
Hence,
\begin{equation*}
A\delta_n+B(h_{n-1}-h_n)\leqslant (A+B+\theta|B|)\delta_n+\frac{|B|}{\theta}h_n+2B\langle x_{n-1}-x_{n-1}^*,x_n^*-x_{n-1}^*\rangle-B\gamma_n^*.`
\end{equation*}
We define $b_n:=\|\lambda(x_{n-1}-x_{n-1}^*)+n(x_n-x_{n-1})\|^2$. By developing the expression of $b_n$ we get that:
\begin{equation*}
\begin{aligned}
b_n&=\|\lambda(x_n-x_n^*)+(n-\lambda)(x_n-x_{n-1})+\lambda(x_n^*-x_{n-1}^*)\|^2\\
&=\|\lambda(x_n-x_n^*)+(n-\lambda)(x_n-x_{n-1})\|^2+\lambda^2\gamma_n^*\\&+2\lambda^2\langle x_n-x_n^*,x_n^*-x_{n-1}^*\rangle+2\lambda(n-\lambda)\langle x_n-x_{n-1},x_n^*-x_{n-1}^*\rangle.
\end{aligned}
\end{equation*}
By applying the following inequality to $u=(n-\lambda)(x_n-x_{n-1})$ and $v=\lambda(x_n-x_n^*)$:
\begin{equation*}
\|u\|^2\leqslant2\|u+v\|^2+2\|v\|^2,
\end{equation*}
it comes that:
\begin{equation*}
\begin{aligned}
(n-\lambda)^2\delta_n&\leqslant 2\|\lambda(x_n-x_n^*)+(n-\lambda)(x_n-x_{n-1})\|^2+2\lambda^2h_n\\
&\leqslant 2b_n+\frac{8\alpha^2}{9}h_n-\Delta_n^*,
\end{aligned}
\end{equation*}
where $\Delta_n^*=2\left(\lambda^2\gamma_n^*+2\lambda^2\langle x_n-x_n^*,x_n^*-x_{n-1}^*\rangle+2\lambda(n-\lambda)\langle x_n-x_{n-1},x_n^*-x_{n-1}^*\rangle\right)$. As $\Delta_n^*\geqslant0$ we get the first claim of the lemma i.e.
\begin{equation}
\forall n>\lambda,\quad \delta_n\leqslant\frac{2}{(n-\lambda)^2}b_n+\frac{8\alpha^2}{9(n-\lambda)^2}h_n.
\end{equation}
This inequality implies that for any $n>\lambda$,
\footnotesize
\begin{equation*}
\begin{aligned}
A\delta_n+B(h_{n-1}-h_n)&\leqslant (|A+B|+\theta|B|)\frac{2}{(n-\lambda)^2}b_n+\left((|A+B|+\theta|B|)\frac{8\alpha^2}{9(n-\lambda)^2}+\frac{|B|}{\theta}\right)h_n\\&+2B\langle x_{n-1}-x_{n-1}^*,x_n^*-x_{n-1}^*\rangle-B\gamma_n^*.
\end{aligned}
\end{equation*}\normalsize
As $F$ satisfies $\mathcal{G}^2_\mu$, we can write that $h_n\leqslant \frac{w_n}{s\mu}$ and thus,
\footnotesize
\begin{equation*}
\begin{aligned}
A\delta_n+B(h_{n-1}-h_n)&\leqslant (|A+B|+\theta|B|)\frac{2}{(n-\lambda)^2}b_n+\left((|A+B|+\theta|B|)\frac{8\alpha^2}{9s\mu(n-\lambda)^2}+\frac{|B|}{s\mu\theta}\right)w_n\\&+2B\langle x_{n-1}-x_{n-1}^*,x_n^*-x_{n-1}^*\rangle-B\gamma_n^*.
\end{aligned}
\end{equation*}\normalsize
By choosing $\theta=\frac{1}{\sqrt{2s\mu}}$ we can conclude that:\footnotesize
\begin{equation*}
\begin{aligned}
A\delta_n+B(h_{n-1}-h_n)&\leqslant \left(2|A+B|+\frac{\sqrt{2}|B|}{\sqrt{s\mu}}\right)\frac{1}{(n-\lambda)^2}b_n+\left(\left(2|A+B|+\frac{\sqrt{2}|B|}{\sqrt{s\mu}}\right)\frac{4\alpha^2}{9s\mu(n-\lambda)^2}+\frac{\sqrt{2}|B|}{\sqrt{s\mu}}\right)w_n\\&+2B\langle x_{n-1}-x_{n-1}^*,x_n^*-x_{n-1}^*\rangle-B\gamma_n^*,
\end{aligned}
\end{equation*}\normalsize
and hence,
\begin{equation*}
\begin{aligned}
A\delta_n+B(h_{n-1}-h_n)\leqslant&\left(2|A+B|+\frac{\sqrt{2}|B|}{\sqrt{s\mu}}\right)\left(1+\frac{4\alpha^2}{9s\mu n^2}\right)\frac{E_n}{(n-\lambda)^2}\\&-B\gamma_n^*+2B\langle x_{n-1}-x_{n-1}^*,x_n^*-x_{n-1}^*\rangle.
\end{aligned}
\end{equation*}\\\qed

\subsection{Proof of Lemma \ref{lem:right-diff}}
\label{sec:proof_right_diff}

Let $\phi'$ denote the derivative of $\phi$ when it is well defined. According to \cite{young1914note}, the function $\phi$ is differentiable except at a countable set of points. This implies that there exists $(t_i)_{i\in\llbracket 1,N\rrbracket}$ and $N\in\N^*\cup\{+\infty\}$ such that for any $i\in \llbracket 0,N-1\rrbracket$ and $t\in (t_i,t_{i+1})$, $\phi'(t)$ is well defined and equal to $\phi_+(t)$. We suppose that the sequence is ordered such that $t_0<t_i< t_{i+1}$ for any $i$ and that $t_N=+\infty$ when $N\neq+\infty$. \\
Suppose that $t\in (t_0,t_1)$. 
\begin{itemize}
\item If $\phi$ is differentiable at $t_0$, then $\phi$ is differentiable on the interval $[t_0,t_1)$ and $\phi'=\phi_+$ in this interval. Consequently inequality \eqref{eq:ineq_phi+} ensures that,
\begin{equation*}
\phi(t)\leqslant \phi(t_0)+\int_{t_0}^t\psi(u)du.
\end{equation*}
\item If $\phi$ is not differentiable at $t_0$, then inequality \eqref{eq:ineq_phi+} guarantees that for $h>0$ sufficiently small,
\begin{equation*}
\phi(t_0+h)\leqslant \phi(t_0)+h\psi(t_0).
\end{equation*}
Then, the previous discussion allows us to say that $\phi$ is differentiable on $[t_0+h, t_1)$. As a consequence, we can say that there exists $H\in(0,t-t_0)$ such that for any $h\in(0,H)$:
\begin{equation*}
\phi(t)\leqslant \phi(t_0+h)+\int_{t_0+h}^t\psi(u)du\leqslant \phi(t_0)+\int_{t_0}^t\psi(u)du+\int_{t_0}^{t_0+h}\left(\psi(t_0)-\psi(u)\right)du.
\end{equation*}
As this inequality is valid for any $h\in(0,H)$, we finally get the wanted inequality \eqref{eq:maj_semidif}.
\end{itemize}
We now suppose that $t=t_1$. We just proved that \eqref{eq:maj_semidif} is true for all $t\in(t_0,t_1)$. Therefore, for all $t\in (t_0,t_1)$,
\begin{equation*}
\phi(t)\leqslant \phi(t_0)+\int_{t_0}^{t_1}\psi(u)du,
\end{equation*}
and as $\phi$ is continuous we get the same inequality at $t=t_1$.\\
By using the same arguments, we can prove that \eqref{eq:maj_semidif} is valid for any $t>t_1$. Indeed, if $t>t_1$, then it means that $t\in (t_i,t_{i+1})$ or that $t=t_i$ for some $i\in\llbracket 1,N\rrbracket$. In both cases, we get the wanted inequality by applying the above reasonings to the consecutive intervals $(t_j,t_{j+1})$ for $0\leqslant j\leqslant i$.\\\qed

\subsection*{Acknowledgements}

This work was supported by PEPR PDE-AI and the ANR Masdol (grant ANR-PRC-CE23). HL acknowledges the financial support of the Ministry of Education, University and Research (grant ML4IP R205T7J2KP).

 \bibliographystyle{abbrv}
\bibliography{ref.bib}

\end{document}